\documentclass{amsart}
\usepackage{amsmath}
\usepackage{color}
\usepackage{mathtools}
\usepackage{amsmath,enumerate,amsfonts,amssymb,color,graphicx,amsthm}

\def\NN{{\mathbb N}}

\def\RR{{\mathbb R}}

\def\SSphere{{\mathbb S}}

\newcounter{marnote}

\numberwithin{equation}{section}

\begin{document}
\newtheorem{thm}{Theorem}[section]
\newtheorem{Def}[thm]{Definition}
\newtheorem{lem}[thm]{Lemma}
\newtheorem{rem}[thm]{Remark}
\newtheorem{question}[thm]{Question}
\newtheorem{prop}[thm]{Proposition}
\newtheorem{cor}[thm]{Corollary}
\newtheorem{example}[thm]{Example}

\title[A Serrin-type  over-determined problem]{A Serrin-type over-determined problem for Hessian equations in the exterior domain}

\author{Bo Wang}
 \address{
School of Mathematics and Statistics, Beijing Institute of Technology, Beijing 100081, China.}
 \email{wangbo89630@bit.edu.cn.}

\author{ Zhizhang Wang}
 \address{
School of Mathematical Sciences, Fudan University, Shanghai, China.}
 \email{zzwang@fudan.edu.cn.}

\date{}

\thanks{The first author is supported by NSFC Grants No.12271028, BNSF Grants No.1222017 and the Fundamental
Research Funds for the Central Universities. The second author is supported by NSFC Grants No.12141105. }

\maketitle

\begin{abstract}
In this paper,  we consider the Hessian equations in some exterior domain with prescribed asymptotic behavior at infinity and Dirichlet-Neumann conditions on its interior boundary. We obtain that there exists a unique bounded domain such that the over-determined problem admits a unique strictly convex solution.
\end{abstract}

\setcounter{section}{0}

\section{Introduction}

 Suppose that $\Omega$ is a smooth, bounded, simply connected, open set in $\mathbb{R}^d$ ($d\geq2$).  Throughout this paper, we always use $\nu$ to denote the unit outer normal of $\partial \Omega$, which is the boundary of $\Omega$. A celebrated theorem of Serrin \cite{serrin1971} states that if $u\in C^{2}(\bar{\Omega})$ satisfies the Poisson equation
\begin{equation*}
\Delta u=d~~~~\mbox{in }\Omega
\end{equation*}
with Dirichlet-Neumann boundary conditions: $u=0$ and $\partial u/\partial \nu=1$ on $\partial\Omega$, then up to a translation, $\Omega$ is a unit ball and $u(x)=(|x|^{2}-1)/2$. Serrin's proof is based on the well-known moving planes method.  After that, Weinberger \cite{w1971} gave another proof by using the integral identities.

Brandolini-Nitsch-Salani-Trombetti \cite{bnst2008a} gave an alternative proof of Serrin's theorem. Meanwhile, they also applied their approach to deal with the over-determined problem for the Hessian equations. More precisely, for $k=1$, $\cdots$, $d$, they proved that if $u\in C^{2}(\bar{\Omega})$ satisfies
\begin{equation*}
\sigma_{k}(\lambda(D^{2}u))=\binom{d}{k}\text{ in } \Omega
\end{equation*}
with Dirichlet-Neumann boundary conditions: $u=0$ and $\partial u/\partial \nu=1$ on $\partial\Omega$, then up to a translation, $\Omega$ is a unit ball and $u(x)=(|x|^{2}-1)/2$. Here $\binom{d}{k}$  denotes the combination number, $\lambda(D^{2}u)=(\lambda_{1},\cdots,\lambda_{d})$ denotes the eigenvalue of $D^{2}u$ and
\begin{equation*}
\sigma_{k}(\lambda(D^{2}u)):=\sum\limits_{1\leq i_{1}<\cdots<i_{k}\leq d}\lambda_{i_{1}}\cdots\lambda_{i_{k}}
\end{equation*}
denotes the $k$-th elementary symmetric function of $\lambda(D^{2}u)$. In particular, if $k=1,d$, this is the Laplace operator and the Monge-Amp\`{e}re operator, respectively.  Later on, in \cite{bnst2008b,bnst2008c}, they further studied the stability of the corresponding problem for  $k=1,d$. The over-determined problem has attracted a lot of attention. Please see \cite{d2023,dz2023,jinxiong,ps1989,qx2017,reichel1995,reichel1996,hs2012,wangbao2014,wangbao2015,wgs1994} and the references therein for more results.

In this paper, we consider a natural question: can we extend the above results to the exterior domain $\mathbb{R}^d\texttt{\symbol{'134}}\bar{\Omega}$ ($d\geq 3$)? More precisely, for $1\leq k\leq d$, we consider the following exterior over-determined problem for the $k$-Hessian equations:
\begin{equation}\label{baozhang3}
\sigma_{k}(\lambda(D^{2}u))=\binom{d}{k}~~~~\text{in } \mathbb{R}^d\setminus \bar{\Omega},
\end{equation}
\begin{equation}\label{baozhang2}
u=0,~~\partial u/\partial \nu=1~~~~\mbox{on }\partial\Omega,
\end{equation}
\begin{equation}\label{ab}
u(x)-\left(\frac{1}{2}x\cdot Ax+b\cdot x+c\right)\rightarrow0~~~~\text{as } |x|\rightarrow \infty,
\end{equation}
where $A$ is a $d\times d$ real symmetric positive definite matrix, $b\in\RR^{d}$, $c\in\RR$ and $``\cdot"$ denotes the inner product of $\mathbb{R}^d$.
\begin{Def}
For $1\leq k\leq d$, let $\mathcal{A}_k$ denote the set of all $d\times d$ real symmetric positive definite matrix $A$ satisfying $\sigma_{k}(\lambda(A))=\binom{d}{k}$.  
\end{Def}
For the Monge-Amp\`{e}re equation, i.e., $k=d$ in (\ref{baozhang3}):
\begin{equation}\label{ma}
\det(D^{2}u)=1~~~~\text{in } \mathbb{R}^d\setminus \bar{\Omega},
\end{equation}
Caffarelli-Li \cite{cl2003} proved that, if $u$ is a convex viscosity solution of (\ref{ma}), then $u$ satisfies the asymptotic behavior (\ref{ab}) at infinity for some $A\in \mathcal{A}_{d}$, $b\in\RR^{d}$ and $c\in\RR$. Moreover, they obtained the decay rate of $u$ approaching the quadratic polynomial at infinity:
\begin{equation}\label{fd4}
\limsup\limits_{|x|\rightarrow\infty}|x|^{d-2}\left|u(x)-\left(\frac{1}{2}x\cdot Ax+b\cdot x+c\right)\right|<\infty.
\end{equation}
With this prescribed asymptotic behavior, they also proved that given any $A\in\mathcal{A}_{d}$, $b\in\RR^{d}$ and $\varphi\in C^{2}(\partial\Omega)$, there exists a constant $c^{*}$, depending on $d$, $\Omega$, $A$, $b$ and $\varphi$, such that the exterior Dirichlet problem: (\ref{ma}), (\ref{ab}) with $c>c^{*}$, and $u=\varphi$ on $\partial\Omega$ admits a unique convex viscosity solution $u\in C^{\infty}(\RR^{d}\setminus\bar{\Omega})\cap C^{0}(\RR^{d}\setminus\Omega)$. After that, Li-Lu \cite{lilu2018} proved that there exists a sharp constant $c_{*}$ such that the existence result holds for $c\geq c_{*}$ while the non-existence result holds for $c<c_{*}$.  

For the Poisson equation, i.e., $k=1$ in (\ref{baozhang3}):
\begin{equation}\label{la}
\Delta u=d~~~~\text{in } \mathbb{R}^d\setminus \bar{\Omega},
\end{equation}
by the classical theorem of B\^{o}cher \cite{bocher}, the asymptotic behavior (\ref{ab}) implies (\ref{fd4}). For the $k$-Hessian equations (\ref{baozhang3}) with $2\leq k\leq d-1$, Bao-Wang \cite{BaoW} proved that the asymptotic behavior (\ref{ab}) implies (\ref{fd4}). In \cite{baolili}, Bao-Li-Li  obtained the existence result of the corresponding exterior Dirichlet problem. Furthermore, Li-Li \cite{lili2018} extended the existence result to the Hessian quotient equations. For more results on the exterior problem of nonlinear elliptic equations, we refer to \cite{AB,baocao,baoli,baolili,baolizhang0,baolizhang,baoxiongzhou,cl2003, dbw2024,liliyuan2020,lilu2018,li2019,reichel1997} and the references therein.

\subsection{Main results}

In this paper, we will investigate the uniqueness and existence of the domain $\Omega$ such that the problem (\ref{baozhang3})-(\ref{ab}) admits a unique strictly convex solution. Note that if $u\in C^{2}(\RR^{d}\setminus\Omega)$ is a strictly convex solution of (\ref{baozhang3}) and (\ref{ab}) for some $A$, $b$, $c$, then $A$ must belong to $\mathcal{A}_k$, see \cite{BaoW}. So we will always assume that $A\in\mathcal{A}_k$ throughout this paper.

For the case $A=I\in\mathcal{A}_k$, where $I$ denotes the $d\times d$ identity matrix, we can give a complete characterization of such domain.
\begin{thm}\label{cases:80}
Let $d\geq 3$ and $1\leq k\leq d$. Given $A=I\in\mathcal{A}_k$ and $b\in\RR^{d}$, let $c_{*}$ be some constant determined by $k,d,b$. Then there exists a unique smooth, bounded domain $\Omega\subset\mathbb{R}^d$ such that the problem \eqref{baozhang3}-\eqref{ab} with $c<c_{*}$   admits a unique strictly convex solution $u\in C^{\infty}(\RR^{d}\setminus\Omega)$. Moreover, $\Omega$ is a ball centered at $-b$ with radius $r_{0}$ depending on $k$, $d$, $b$, $c$, and $u$ is radially symmetric. Additionally, the problem \eqref{baozhang3}-\eqref{ab} with $c\geq c_{*}$ admits no strictly convex solution for any smooth bounded domain.
\end{thm}

\begin{rem}
The explicit expression of  $c_{*}$ can be given by 
\begin{equation*}
c_*=\left(\frac{d-k}{d}\right)^{2/k}\left[-\frac{1}{2}+\int^{\infty}_1\tau\left(\left(1+\frac{k}{d-k}\tau^{-d}\right)^{1/k}-1\right)d\tau\right]+\frac{1}{2}|b|^{2}
\end{equation*}
for $1\leq k\leq d-1$ and  $$c_*=\int_{0}^{\infty}\tau((1+\tau^{-d})^{1/d}-1)d\tau+\frac{1}{2}|b|^{2}~~~~~~~~\mbox{                for $k=d$}.$$ 
\end{rem}

Compared with the results obtained by Serrin \cite{serrin1971} and Brandolini-Nitsch-Salani-Trombetti \cite{bnst2008a}, for the case $A\in\mathcal{A}_k$ and $A\neq I$, one may expect the non-existence of the domain such that the problem \eqref{baozhang3}-\eqref{ab} admits a unique strictly convex solution. However, surprisingly, we can construct some non-symmetric domain by performing Legrendre transform and solving the corresponding dual problem. This phenomenon is quite different from the Serrin-type over-determined problem in the bounded domain. This is the main novelty of this paper.

Denote $\lambda_{\min}(A)$ and $\lambda_{\max}(A)$ to be the minimum eigenvalue and maximum eigenvalue of $A$, respectively. Let 
$$\bar{c}=\bar{c}(A,b):= \frac{|b|^{2}\lambda_{\min}(A)-\lambda_{\max}(A)}{2\lambda_{\min}(A)\lambda_{\max}(A)}.$$ Then we have the following theorems.

\begin{thm}\label{cases:10}
Let $d\geq 3$ and $2\leq k\leq d$. Given $(A,b,c)\in\mathcal{A}_{k}\times \RR^{d}\times(-\infty,\bar{c})$, there exists a unique smooth, bounded domain $\Omega\subset\mathbb{R}^d$ such that the problem \eqref{baozhang3}-\eqref{ab} admits a unique strictly convex solution $u\in C^{\infty}(\RR^{d}\setminus\Omega)$.
\end{thm}

For the Poisson equation, i.e. equation \eqref{la}, we need an extra assumption on $A$.

\begin{thm}\label{cases:20}
Let $d\geq 3$. Given $(A,b,c)\in\mathcal{A}_{1}\times \RR^{d}\times(-\infty,\bar{c})$, if we further assume that  $\lambda_{\max}(A)<\frac{1}{2}$, then there exists a unique smooth, bounded domain $\Omega\subset\mathbb{R}^d$ such that the problem \eqref{la}, \eqref{baozhang2}, \eqref{ab} admits a unique strictly convex solution $u\in C^{\infty}(\RR^{d}\setminus\Omega)$.
\end{thm}

\begin{rem}\label{sm}
The decay rate for the solutions $u$, which are obtained by the above two theorems, approaching the quadratic polynomial is \eqref{fd4}. 

\end{rem}

\begin{rem}
Concerning the above two theorems, it may be an  interesting question to find  the sharp constant $c_{*}$ for the case $A\in\mathcal{A}_k$ and $A\neq I$ such that the problem \eqref{baozhang3}-\eqref{ab} admits a unique strictly convex solution which is non-radially symmetric.
\end{rem}

Note that our overdetermined problem for the Poisson equation is indeed a free boundary problem, which may be related to Berestycki-Caffarelli-Nirenberg conjecture \cite{BCN}. For more results on the free boundary problem, we refer to \cite{cs} and the references therein.

\subsection{Questions and comments} 

A function $u \in C^2 (\RR^{d}\setminus\bar{\Omega})$ is called $k$-convex if 
\begin{equation*}
\lambda (D^2 u) \in \Gamma_k:= \{\lambda \in \mathbb{R}^{d}: \sigma_{j} (\lambda) > 0, j = 1, \cdots, k\}~~~~\mbox{in }\RR^{d}\setminus\bar{\Omega}.
\end{equation*}
\begin{Def}
For $1\leq k\leq d$, let $\tilde{\mathcal{A}}_k$ denote the set of all $d\times d$ real symmetric matrix $A$ satisfying $\lambda (A) \in \Gamma_k$ and $\sigma_{k}(\lambda(A))=\binom{d}{k}$.  
\end{Def}

\begin{lem}\label{general}
For $1\leq k\leq d-1$, if $u\in C^{2}(\RR^{d}\setminus\bar{\Omega})$ is a $k$-convex solution of (\ref{baozhang3}) and satisfies (\ref{ab}) for some $d\times d$ real symmetric matrix $A$, $b\in\RR^{d}$ and $c\in\RR$, then $A\in\tilde{\mathcal{A}}_k$.
\end{lem}

For $1\leq k\leq d-1$, it is natural to discuss the uniqueness and existence of  the domain $\Omega$ such that the problem (\ref{baozhang3})-(\ref{ab}) admits a unique solution in the class of $k$-convex functions.

For the uniqueness, we propose the following question:

\begin{question}\label{ershu}
Let $d\geq 3$ and $1\leq k\leq d-1$. Given $(A,b,c)\in\tilde{\mathcal{A}}_{k}\times \RR^{d}\times\RR$, suppose that $(\Omega_{1}$, $u_{1})$ and $(\Omega_{2}$, $u_{2})$ are two pairs of domain and solution such that the problem \eqref{baozhang3}-\eqref{ab} admits a unique solution in the class of $k$-convex functions. Can we conclude that $\Omega_{1}=\Omega_{2}$ and $u_{1}\equiv u_{2}$?
\end{question}

\begin{rem}
As in the proof of our main results, the answer of the above question is Yes in the class of strictly convex functions. 
\end{rem}

For the existence, we propose the following question:  
\begin{question}\label{bashu}
Let $d\geq 3$ and $1\leq k\leq d-1$. Given $(A,b,c)\in\tilde{\mathcal{A}}_{k}\times \RR^{d}\times\RR$, does there exist a bounded domain such that the problem \eqref{baozhang3}-\eqref{ab} admits a unique $k$-convex solution?
\end{question}

For the case $A=I\in\tilde{\mathcal{A}}_{k}$, we can find a constant $\hat{c}$ such that the answer of Question \ref{bashu} is Yes with $c<\hat{c}$ and the domain is a ball. 

\begin{rem}\label{pangxie}
Let $d\geq 3$ and $1\leq k\leq d-1$. Given $A=I\in\tilde{\mathcal{A}}_{k}$ and $b\in\RR^{d}$, suppose that $\Omega$ is a ball centered at $-b$ with some radius $r_{0}>0$. Then there exists a constant $\hat{c}=\hat{c}(k,d,b)$ such that, if $c\geq \hat{c}$, the problem \eqref{baozhang3}-\eqref{ab} admits no $k$-convex solution and if $c<\hat{c}$, the problem \eqref{baozhang3}-\eqref{ab} admits a unique $k$-convex solution $u\in C^{\infty}(\RR^{d}\setminus\Omega)$ of the form
\begin{equation}\label{daxia}
u(x)=\int_{r_{0}}^{|x+b|}s(1+Cs^{-d})^{1/k}ds,
\end{equation}
where $C=C(r_{0})=r_{0}^{d-k}-r_{0}^{d}$. Moreover, the constant $\hat{c}$ is strictly decreasing with respect to $k$. 
\end{rem}

In view of Remark \ref{pangxie}, if $A=I\in\tilde{\mathcal{A}}_{k}$, one may ask Question \ref{ershu} more precisely. That is the following rigidity question.  
\begin{question}\label{laosan}
Let $d\geq 3$ and $1\leq k\leq d-1$. Given $A=I\in\tilde{\mathcal{A}}_{k}$, $b\in\RR^{d}$ and $c\in\RR$, suppose that there exists a bounded domain $\Omega$ such that the problem \eqref{baozhang3}-\eqref{ab} admits a unique $k$-convex solution $u\in C^{\infty}(\RR^{d}\setminus\Omega)$, then can we conclude that $\Omega$ is a ball centered at $-b$ with some radius $r_{0}>0$ and $u$ is of the form \eqref{daxia}?
\end{question}

We can give a partial answer of Question \ref{laosan} under an additional assumption:
\begin{equation}\label{fd4--}
\limsup\limits_{|x|\rightarrow\infty}|x|^{\alpha}\left|u(x)-\left(\frac{1}{2}|x|^{2}+c\right)\right|<\infty,
\end{equation}
where $\alpha>d$. 

\begin{thm}\label{ridig''}
Let $d\geq 3$ and $2\leq k\leq d-1$. Suppose that there exists a bounded star-shaped domain $\Omega\subset\mathbb{R}^d$ such that problem \eqref{baozhang3}, \eqref{baozhang2}, \eqref{fd4--} with some $c<0$ and $\alpha>d$ admits a $k$-convex solution $u\in C^{2}(\RR^{d}\setminus\Omega)$, then $\Omega$ is a unit ball and $u(x)=\frac{1}{2}|x|^{2}-\frac{1}{2}$.
\end{thm}

For $k=1$, the star-shapedness and the condition $c< 0$ can be removed.  
\begin{thm}\label{ridig}
Let $d\geq 3$. Suppose that there exists a bounded domain $\Omega\subset\mathbb{R}^d$ such that the problem \eqref{la}, \eqref{baozhang2}, \eqref{fd4--} with some $c\in\RR$ and $\alpha>d$ admits a solution $u\in C^{2}(\RR^{d}\setminus\Omega)$, then $\Omega$ is a unit ball and $u(x)=\frac{1}{2}|x|^{2}-\frac{1}{2}$.
\end{thm}

\subsection{Outline of the proof}
We give a brief introduction of our proof. 
By an orthogonal transform and a translation, we may assume that 
\begin{equation*}
A=\mbox{diag}\{1/a_{1},\cdots,1/a_{d}\}~~~~\mbox{and}~~~~b=0,
\end{equation*}
where $a_{i}>0$, $i=1$, $\cdots$, $d$, and $\sigma_{k}(1/a_{1},\cdots,1/a_{d})=\binom{d}{k}$. Then the asymptotic behavior (\ref{fd4}) can be written as
\begin{equation}\label{fd4'}
\limsup\limits_{|x|\rightarrow\infty}|x|^{d-2}\left|u(x)-\left(\frac{1}{2}\sum\limits_{i=1}^{d}\frac{x_{i}^{2}}{a_{i}}+c\right)\right|<\infty.
\end{equation}

Let $u^{*}=u^{*}(p)$ denote the Ledgendre transform function of $u$. Then, in Section 2, we will show that $u^{*}$ satisfies the Hessian quotient equation in the exterior domain:
\begin{equation}\label{fd5}
\sigma_{d}(\lambda(D^{2}u^{*}))/\sigma_{d-k}(\lambda(D^{2}u^{*}))=\binom{d}{k}~~~~\mbox{in }\mathbb{R}^d\texttt{\symbol{'134}}\bar{B}_{1}
\end{equation}
with the Robin boundary condition on its interior boundary:
\begin{equation}\label{fd6}
u^{*}=\partial u^{*}/\partial n~~~~\mbox{on }\partial B_{1},
\end{equation}
and the following prescribed asymptotic behavior at infinity:
\begin{equation}\label{fd7}
\limsup\limits_{|p|\rightarrow\infty}|p|^{d-2}\left|u^{*}(p)-\left(\frac{1}{2}\sum\limits_{i=1}^{d}a_{i}p_{i}^{2}-c\right)\right|<\infty,
\end{equation}
where $B_{1}\subset\RR^{d}$ denotes the unit ball centered at the origin and $n=n(p)=p$ for any $p\in\partial B_{1}$.

The problem (\ref{fd5})-(\ref{fd7}) is a Hessian quotient equation with the oblique boundary condition on its interior boundary and the dual asymptotic behavior at infinity.  In order to solve this problem, we will divide our proof into three steps.

Step 1. For $0\leq t\leq 1$, in Section 3, we will construct the subsolution $\underline{\omega}_t$ and the supsolution $\bar{\omega}_t$ of the problem: (\ref{fd5}) with the right hand side $\binom{d}{k}$ replaced by $\sigma_{d}(a(t))/\sigma_{d-k}(a(t))$, (\ref{fd6}) and the following asymptotic behavior at infinity:
\begin{equation}\label{fd7----------}
\limsup\limits_{|p|\rightarrow\infty}|p|^{d_{k}^{*}(a(t))-2}\left|u^{*}(p)-\left(\frac{1}{2}\sum\limits_{i=1}^{d}a_{i}(t)p_{i}^{2}-c\right)\right|<\infty,
\end{equation}
where 
\begin{equation}\label{at}
a(t):=(a_1(t),\cdots, a_d(t))=ta+(1-t)(1,\cdots,1),
\end{equation}
and
\begin{equation*}
d_{k}^{*}(a(t)):=\frac{k}{\max\limits_{1\leq i\leq d}a_{i}(t)\sigma_{k-1;i}(a(t))}.
\end{equation*}

Step 2. For sufficiently large $R>1$, in Section 4, we will consider the following problem on the annulus:
\begin{equation}\label{fd5a}
\sigma_{d}(\lambda(D^{2}u^{*}_{R,t}))/\sigma_{d-k}(\lambda(D^{2}u^{*}_{R,t}))=\sigma_{d}(a(t))/\sigma_{d-k}(a(t))~~~~\mbox{in }B_{R}\texttt{\symbol{'134}}\bar{B}_{1}
\end{equation}
with the Robin boundary condition:
\begin{equation}\label{fd6a}
u^{*}_{R,t}=\partial u^{*}_{R,t}/\partial n+(1-t)\psi_t~~~~\mbox{on }\partial B_{1},
\end{equation}
and the Dirichlet boundary condition:
\begin{equation}\label{fd7a}
u^{*}_{R,t}=\underline{\omega}_t~~~~\mbox{on }\partial B_{R},
\end{equation}
where $\underline{\omega}_t$ is constructed as in Step 1 and 
\begin{equation}\label{psit}
\psi_t:=\underline{\omega}_t-\partial \underline{\omega}_t/\partial n.
\end{equation}
It is easy to see that the above problem is solvable for $t=0$ with $u^{*}_{R,0}\equiv\underline{\omega}_0$. In order to solve the above problem for $t=1$, by the continuity method, we will derive the closedness and openness of the problem (\ref{fd5a})-(\ref{fd7a}) with respect to $t$. 

Step 3. For sufficiently large $R>1$, let $u_{R,1}^{*}$ be the solution of (\ref{fd5a})-(\ref{fd7a}) for $t=1$. In Section 5, we will derive the $C^{0}$-$C^{2}$ local estimates of  $u_{R,1}^{*}$ with respect to $R$. Then, using the Cantor subsequences method, we obtain that $u_{R,1}^{*}$ converges to a function $u^{*}$, which is the unique strictly convex solution of the problem (\ref{fd5})-(\ref{fd7}). Since the Pogorelov type interior estimates does not always hold for the Hessian quotient equations, we cannot obtain the asymptotic behavior directly for the dual problem. To overcome this difficulty, we use the Pogorelov type interior estimates for the Hessian equation  \eqref{baozhang3} to obtain the asymptotic behavior, then we derive the asymptotic behavior of the dual problem via the Legendre transform.      

By the above three steps, we can finish the proof of Theorem \ref{cases:10} and \ref{cases:20}.

The rest of this paper is organized as follows. In Section 6, we give the proof of Theorem 1.2 and Remark 1.13. In Section 7, we give the proof of Theorem 1.15 and 1.16. In Section 8, we give the proof of Lemma 1.9.

\section{The dual problem}

Let $u$ be a locally strictly convex $C^{2}$ function defined in $\mathbb{R}^d\verb"\"\Omega$. We define $u^{*}$, the Ledgendre transform function of $u$ as
 \begin{equation}
 u^{*}(p)=p\cdot x-u(x), \quad \forall~x\in\mathbb{R}^d\verb"\"\Omega,~p=Du(x)\in\bar{\Omega}^{*},\label{cases:11}
  \end{equation}
where $\Omega^{*}:=Du(\mathbb{R}^d\texttt{\symbol{'134}}\bar{\Omega})$.

With the aid of Ledgendre transform, we have the following lemma.

\begin{lem}\label{miyun}
If $u\in
C^{2}(\mathbb{R}^d\verb"\"\Omega)$ is a locally strictly convex solution of \eqref{baozhang3}, \eqref{baozhang2} and \eqref{fd4'}, then the Legendre transform of $u$, $u^{*}\in
C^{2}(\mathbb{R}^d\verb"\"B_{1})$ is a locally strictly convex solution of \eqref{fd5}-\eqref{fd7}.

\end{lem}

\begin{proof}

We firstly prove that
\begin{equation}\label{dd}
\Omega^{*}=\mathbb{R}^d\verb"\"\bar{B}_{1}.
\end{equation}
Indeed, on one side, for any $p\in\Omega^{*}$, there exists $x=x(p)\in\mathbb{R}^d\texttt{\symbol{'134}}\bar{\Omega}$ such that $p=Du(x)$. By (\ref{baozhang2}) and the strictly convexity of $\Omega$, there exist $x_{0}\in\partial\Omega$ and $\xi_{0}\in\mathbb{R}^d\texttt{\symbol{'134}}\bar{\Omega}$ such that
  \begin{equation*}
  |p|=|Du(x)|\geq\frac{\partial u}{\partial \nu}(x)=\frac{\partial u}{\partial \nu}(x_{0})+\frac{\partial^{2} u}{\partial \nu^{2}}(\xi_{0})|x-x_{0}|>1.
   \end{equation*}
Thus, we have that
  \begin{equation}
  \Omega^{*}\subset \mathbb{R}^d\verb"\"\bar{B}_{1}.\label{cases:3}
   \end{equation}
  On the other side, for any $p\in\partial\Omega^{*}$,  there exists a sequence $\{p^{(m)}\}_{m=1}^{\infty}\subset\Omega^*$ such that $\lim\limits_{m\rightarrow\infty}p^{(m)}=p$. Let $x^{(m)}\in\RR^{d}\setminus\bar{\Omega}$ such that $p^{(m)}=Du(x^{(m)})$ for $m\in\NN^{+}$. By Lemma 3.2 in \cite{BaoW}, we have that, for $i=1$, $\cdots$, $d$, 
   \begin{equation}\label{Du}
   D_{i}u(x)=\frac{x_{i}}{a_{i}}+O\left(|x|^{1-d}\right),~~~~\mbox{as $|x|\rightarrow\infty$}.
     \end{equation}
It follows that $\{x^{(m)}\}_{m=1}^{\infty}$ is a bounded sequence in  $\RR^{d}\setminus\bar{\Omega}$. Thus, there exists some subsequence, which converges to some $x\in\RR^{d}\setminus\Omega$ and $p=Du(x)$.   If  $x\in\RR^{d}\setminus\bar{\Omega}$,  since $Du$ maps interior point to interior point, we know that $p$ is an interior point, which is a contradiction. Thus $x\in\partial\Omega$. By \eqref{baozhang2}, we can conclude that 
    \begin{equation}
  \partial\Omega^{*}\subset Du(\partial\Omega)\subset \partial B_{1}.\label{cases:4}
   \end{equation}
Then (\ref{cases:3}) and (\ref{cases:4}) yields (\ref{dd}).

We secondly prove that $u^{*}$ is a $C^{2}$ locally strictly convex solution of (\ref{fd5})-(\ref{fd6}). Indeed, by the definition of $u^{*}$ and (\ref{dd}), $u^{*}\in
C^{2}(\mathbb{R}^d\verb"\"B_{1})$ is a locally strictly convex function. Moreover, we have that $D^{2}u^{*}=(D^{2}u)^{-1}$ in $\RR^{d}\setminus\bar{B}_{1}$. Consequently, by (\ref{baozhang3}), we can conclude that $u^{*}$ satisfies (\ref{fd5}). By (\ref{baozhang2}), for any $p\in \partial B_{1}$, there exists $x=x(p)\in\partial\Omega$ such that $p=Du(x)$ and
\begin{equation*}
   u^{*}(p)=p\cdot x-u(x)=p\cdot Du^{*}(p)=\frac{\partial u^{*}}{\partial n}(p),
\end{equation*}
that is (\ref{fd6}).

We thirdly prove that $u^{*}$ satisfies (\ref{fd7}). By \eqref{Du}, we have that, for $i=1$, $\cdots$, $d$,
   \begin{equation*}
x_{i}=a_{i}p_{i}+O\left(|p|^{1-d}\right),~~~~\mbox{as $|p|\rightarrow\infty$},
\end{equation*}
where we have used the fact that $O(|x|^{-1})=O(|p|^{-1})$. It follows from (\ref{fd4'}) that 
  \begin{equation*}
 u^{*}(p)=p\cdot x-u(x)=\frac{1}{2}\sum\limits_{i=1}^{d}a_{i}p_{i}^{2}-c+O\left(|p|^{2-d}\right),~~~~\mbox{as }|p|\rightarrow\infty,\label{cases:5}
   \end{equation*}
that is (\ref{fd7}).

\end{proof}

\section{The Sub and Super Solutions}

In this section, for any vector $a=(a_1,\cdots,a_d)\in\mathbb{R}^d$ with $a_i>0$, $i=1$, $\cdots$, $d$, we will construct the strictly convex sub and super solutions of the equation
\begin{equation}\label{fd5-------}
\sigma_{d}(\lambda(D^{2}u^{*}))/\sigma_{d-k}(\lambda(D^{2}u^{*}))=\sigma_d(a)/\sigma_{d-k}(a)~~~~\mbox{in }\mathbb{R}^d\texttt{\symbol{'134}}\bar{B}_{1}
\end{equation}
with the boundary condition (\ref{fd6}) and the following prescribed asymptotic behavior at infinity:
\begin{equation}\label{fd7'}
\limsup\limits_{|p|\rightarrow\infty}|p|^{d_k^*(a)-2}\left|u^{*}(p)-\left(\frac{1}{2}\sum\limits_{i=1}^{d}a_{i}p_{i}^{2}-c\right)\right|<\infty.
\end{equation}

We construct the subsolutions of the form
\begin{equation*}
\underline{\omega}(p)=\Phi(r),~~~~r:=\sqrt{\sum\limits_{i=1}^{d}a_{i}p_{i}^{2}}.
\end{equation*}
A straightforward calculation yields that for $i$, $j=1$, $\cdots$, $d$,
\begin{align*}
D_{ij}\underline{\omega}&=a_{i}r^{-1}\Phi'\delta_{ij}+r^{-2}(\Phi''-r^{-1}\Phi')(a_{i}p_{i})(a_{j}p_{j})\\
&=a_{i}h\delta_{ij}+r^{-1}h'(a_{i}p_{i})(a_{j}p_{j}),
\end{align*}
where $h:=\Phi'/r$. Denote $l=d-k$. By Proposition 1.2 in \cite{baolili}, we have that
\begin{align*}
\frac{\sigma_{d}(\lambda(D^{2}\underline{\omega}))}{\sigma_{l}(\lambda(D^{2}\underline{\omega}))}&=\frac{\sigma_{d}(a)}{\sigma_{l}(a)}\frac{h^{d}+r^{-1}h'h^{d-1}\sum\limits_{i=1}^{d}\frac{\sigma_{d-1;i}(a)a_{i}}{\sigma_{d}(a)}(a_{i}p_{i}^{2})}{h^{l}+r^{-1}h'h^{l-1}\sum\limits_{i=1}^{d}\frac{\sigma_{l-1;i}(a)a_{i}}{\sigma_{l}(a)}(a_{i}p_{i}^{2})}\\
&=\frac{\sigma_{d}(a)}{\sigma_{l}(a)}\frac{h^{d}+rh'h^{d-1}}{h^{l}+r^{-1}h'h^{l-1}\sum\limits_{i=1}^{d}\frac{\sigma_{l-1;i}(a)a_{i}}{\sigma_{l}(a)}(a_{i}p_{i}^{2})}.
\end{align*}
Let $\underline{t}_{0}(a):=0$ and 
\begin{equation*}
\underline{t}_{l}(a):=\min\limits_{1\leq i\leq d}\frac{\sigma_{l-1;i}(a)a_{i}}{\sigma_{l}(a)},~~~~l=1,\cdots,d-1.
\end{equation*}
Note that $0\leq\underline{t}_{l}(a)\leq\frac{l}{d}<1$. We will consider the ODE
\begin{equation}\label{chenglu'}
\frac{h^{d}+rh'h^{d-1}}{h^{l}+\underline{t}_{l}(a)rh'h^{l-1}}=1.
\end{equation}
We rewrite the above equation into the following divergence form:
\begin{equation}\label{chenglu}
\left(r^{\frac{d-l}{1-\underline{t}_{l}(a)}}h^{\frac{d-l}{1-\underline{t}_{l}(a)}}\right)'=\left(r^{\frac{d-l}{1-\underline{t}_{l}(a)}}h^{\frac{d-l}{1-\underline{t}_{l}(a)}\underline{t}_{l}(a)}\right)'.
\end{equation}
It is easy to see that 
\begin{equation}\label{doujiali}
d^*_k(a)=\frac{d-l}{1-\underline{t}_{l}(a)}.
\end{equation}
Indeed, let $\lambda_{i}:=1/a_{i}$, $i=1$, $\cdots$, $d$. We use $1/\lambda$ to denote the vector $a$. Then we have that 
\begin{equation*}
\underline{t}_{l}(a)=\min\limits_{1\leq i\leq d}\frac{\sigma_{l-1;i}(\frac{1}{\lambda})\frac{1}{\lambda_{i}}}{\sigma_{l}(\frac{1}{\lambda})}=\min\limits_{1\leq i\leq d}\frac{\sigma_{k;i}(\lambda)}{\sigma_{d-1;i}(\lambda)\lambda_{i}}\frac{\sigma_{d}(\lambda)}{\sigma_{k}(\lambda)}=\min\limits_{1\leq i\leq d}\frac{\sigma_{k;i}(\lambda)}{\sigma_{k}(\lambda)}.
\end{equation*}
It follows that 
\begin{equation*}
1-\underline{t}_{l}(a)=\max\limits_{1\leq i\leq d}\left(1-\frac{\sigma_{k;i}(\lambda)}{\sigma_{k}(\lambda)}\right)=\max\limits_{1\leq i\leq d}\frac{\sigma_{k-1;i}(\lambda)\lambda_{i}}{\sigma_{k}(\lambda)}=\frac{k}{d_{k}^{*}(a)},
\end{equation*}
that is (\ref{doujiali}).

By (\ref{doujiali}), the equation (\ref{chenglu}) can be written as 
\begin{equation*}
\left(r^{d^*_k(a)}h^{d^*_k(a)}\right)'=\left(r^{d^*_k(a)}h^{d^*_k(a)\underline{t}_{l}(a)}\right)'.
\end{equation*}
Integrating the above equation and dividing by $r^{-d_{k}^{*}(a)}$, we have that
 \begin{equation}\label{xi}
\xi(h):=h^{d^*_k(a)}-h^{d^*_k(a)\underline{t}_{l}(a)}=C_{1}r^{-d^*_k(a)},
\end{equation}
where $C_{1}$ is an arbitrary constant. It is easy to see that $\xi$ is  strictly decreasing on $\left[0,\left(\underline{t}_{l}(a)\right)^{\frac{1}{d-l}}\right]$ and strictly increasing on $\left[\left(\underline{t}_{l}(a)\right)^{\frac{1}{d-l}},\infty\right)$.

We take $C_{1}\geq0$. Since $\xi(0)=0$ and $\xi\rightarrow \infty$ as $h\rightarrow\infty$, then we can solve $h$ from (\ref{xi}) as 
$$h(r)=\xi^{-1}\left(C_{1}r^{-d^*_k(a)}\right),$$ where $\xi^{-1}$ denotes the inverse of $\xi$ on $\left[\left(\underline{t}_{l}(a)\right)^{\frac{1}{d-l}},\infty\right)$. Moreover, $h\geq1$ and $h'\leq0$. Recall that $h=\Phi'/r$. Let 
\begin{equation}\label{w}
\underline{\omega}(p):=\Phi(r,\eta)=\int^{r}_{\eta}s\xi^{-1}\left(C_{1}s^{-d^*_k(a)}\right)ds+\eta^2\xi^{-1}(C_1\eta^{-d^*_k(a)}),
\end{equation}
where $\eta=\min\limits_{1\leq i\leq d}\sqrt{a}_{i}$.

Now we prove that $\underline{\omega}$ is a strictly convex subsolution of (\ref{fd5-------}) satisfying (\ref{fd6}).
By \eqref{chenglu'}, $0\leq\underline{t}_{l}(a)<1$ and $h\geq1$, we have that 
\begin{equation}\label{shengguo}
h'=-\frac{h}{r}\frac{h^{d-1}-h^{l-1}}{h^{d-1}-\underline{t}_{l}(a)h^{l-1}}>-\frac{h}{r}.
\end{equation}
It follows from the above inequality, $0\leq \underline{t}_{l}(a)<1$, $h\geq1$ and $h'\leq0$ that 
\begin{align*}
\sigma_{j}\left(\lambda(D^{2}\underline{\omega})\right)&=\sigma_{j}(a)\left(h^{j}+r^{-1}h'h^{j-1}\sum\limits_{i=1}^{d}\frac{\sigma_{j-1;i}(a)a_{i}}{\sigma_{j}(a)}(a_{i}p_{i}^{2})\right)\\
&\geq \sigma_{j}(a)(h^{j}+\underline{t}_{j}(a)rh'h^{j-1})\\
&\geq\sigma_{j}(a)h^{j-1}(h+rh')>0,~~~~j=1,\cdots,d,
\end{align*}
and
\begin{eqnarray*}
\frac{\sigma_{d}(\lambda(D^{2}\underline{\omega}))}{\sigma_{l}(\lambda(D^{2}\underline{\omega}))}&=&\frac{\sigma_{d}(a)}{\sigma_{l}(a)}\frac{h^{d}+r^{-1}h'h^{d-1}\sum\limits_{i=1}^{d}\frac{\sigma_{d-1;i}(a)a_{i}}{\sigma_{d}(a)}(a_{i}p_{i}^{2})}{h^{l}+r^{-1}h'h^{l-1}\sum\limits_{i=1}^{d}\frac{\sigma_{l-1;i}(a)a_{i}}{\sigma_{l}(a)}(a_{i}p_{i}^{2})}\\
&\geq & \frac{\sigma_{d}(a)}{\sigma_{l}(a)}\frac{h^{d}+rh'h^{d-1}}{h^{l}+\underline{t}_{l}(a)rh'h^{l-1}}=\frac{\sigma_{d}(a)}{\sigma_{l}(a)}.
\end{eqnarray*}
On the ellipse $E_{\eta}=\{p\in\RR^{d}:\sum\limits_{i=1}^{d}a_ip_i^2=\eta^2\}$, we have that
\begin{equation*}
\underline{\omega}-p\cdot D\underline{\omega}=\Phi-r\Phi'=0.
\end{equation*}
Since 
$$(\Phi-r\Phi')'=-r\Phi''=-r(h+rh')< 0,~~~~\forall~r\geq \eta,$$
in view of $E_{\eta}\subset B_1$, we have that
$$\underline{\omega}-\partial\underline{\omega}/\partial n\leq 0\ \ \ \text{ on } \partial B_1,$$
which implies that $\underline{\omega}$ is a strictly convex subsolution of (\ref{fd5-------}) and (\ref{fd6}).

To discuss the asymptotic behavior, we rewrite \eqref{w} as
\begin{equation}\label{sub}
\underline{\omega}(p)=\frac{1}{2}\sum\limits_{i=1}^{d}a_{i}p_{i}^{2}+\mu(C_{1})-\int_r^{\infty}s\left[\xi^{-1}(C_{1}s^{-d^*_k(a)})-1\right] ds,
\end{equation}
where
\begin{equation*}
\mu(C_{1}):=\int_{\eta}^{\infty}s\left[\xi^{-1}(C_{1}s^{-d^*_k(a)})-1\right] ds-\frac{\eta^2}{2}+\eta^2\xi^{-1}(C_1\eta^{-d^*_k(a)}).
\end{equation*}
Since $\xi^{-1}(0)=1$, we have that as $s\rightarrow \infty$,
$$\xi^{-1}(C_{1}s^{-d^*_k(a)})-1=\xi^{-1}(C_{1}s^{-d^*_k(a)})-\xi^{-1}(0)=O(s^{-d^*_k(a)}).$$
Thus, if $d^*_k(a)>2$, we have that as $r\rightarrow \infty$,
$$\int_r^{\infty}s\left[\xi^{-1}(C_{1}s^{-d^*_k(a)})-1\right] ds=O(r^{2-d^*_k(a)}).$$ 
Inserting the above equality into (\ref{sub}), we can obtain the asymptotic behavior of $\underline{\omega}$ at infinity:
\begin{equation}\label{asyw}
\underline{\omega}(p)=\frac{1}{2}\sum\limits_{i=1}^{d}a_{i}p_{i}^{2}+\mu(C_{1})+O(|p|^{2-d^*_k(a)}),~~~~\mbox{as }p\rightarrow\infty.
\end{equation}
Note that the condition $d_k^*(a)>2$ is equivalent to 
$$\underline{t}_{d-k}(a)>-(k-2)/2.$$
For $k\geq 2$, it obviously holds. For $k=1$,  the above inequality becomes
\begin{equation*}
\sigma_{d-1;i}(a)<\frac{1}{2}\sigma_{d-1}(a),~~~~i=1,\cdots, d.
\end{equation*}

Note that $\mu(C_{1})$ is strictly increasing with respect to $C_{1}$, $\mu(0)=\eta^2/2$ and $\lim\limits_{C_{1}\rightarrow\infty}\mu(C_{1})=\infty$. Then we have that the range of $\mu$ for $C_1\geq 0$ is $\left[\frac{\eta^2}{2},\infty\right)$ . Therefore, we need to require that $-c\geq \eta^2/2=\frac{1}{2}\left(\min\limits_{1\leq i\leq d}a_{i}\right)$.

The super solutions can be taken directly as
\begin{equation}\label{sup}
\bar{\omega}(p)=\frac{1}{2}\sum\limits_{i=1}^{d}a_{i}p_{i}^{2}-c,
\end{equation}
where $-c\geq\frac{1}{2}\left(\max\limits_{1\leq i\leq d}a_{i}\right)$. Indeed, it is easy to see that $\bar{\omega}$ is a strictly convex solution of (\ref{fd5-------}) satisfying (\ref{fd7'}). Moreover, we have that
\begin{equation*}
\bar{\omega}-\partial\bar{\omega}/\partial n=-c-\frac{1}{2}\sum\limits_{i=1}^{d}a_{i}p_{i}^{2}\geq -c-\frac{1}{2}\left(\max\limits_{1\leq i\leq d}a_{i}\right)\geq 0~~~~\mbox{on }\partial B_{1}.
\end{equation*}

\section{The Approximate Problem}

For $0\leq t\leq 1$, let $a(t)$ be defined as in (\ref{at}) and let $\underline{\omega}_t$ be constructed as in \eqref{w} with $a$ replaced by $a(t)$. 

In this section, our goal is to solve (\ref{fd5a})-(\ref{fd7a}) for $t=1$. Since the problem is obviously solvable for $t=0$,  by the continuity method, we need to establish the closedness and the openness of (\ref{fd5a})-(\ref{fd7a}).

The closedness can be obtained by Lemma \ref{c0}-\ref{c2} below.

 \begin{lem}\label{c0}
 Let $R>1$. There exists a constant $C>0$ depending on $R$, $a$ and $c$ such that for any $0\leq t\leq 1$ and any $u_{R,t}^{*}\in C^{2}(\bar{B}_{R}\setminus B_{1})$ satisfying \eqref{fd5a}-\eqref{fd7a}, we have that 
 \begin{equation}\label{c0estimate}
 |u_{R,t}^{*}|\leq C~~~~\mbox{on }\bar{B}_{R}\setminus B_{1}.
 \end{equation}
  \end{lem}
  
  \begin{proof}
  Let $\underline{\omega}_t$ and $\bar{\omega}_t$ be the sub and super solution constructed as in \eqref{w} and \eqref{sup} with $a$ replaced by $a(t)$. Then, by the maximum principle, we have that 
   \begin{equation*}
 \underline{\omega}_t\leq u^{*}_{R,t}\leq \bar{\omega}_t~~~~\mbox{on }\bar{B}_{R}\setminus B_{1}.
 \end{equation*}
 (\ref{c0estimate}) follows from the above inequality directly.
 \end{proof}

 In the following, we will derive the estimates of $Du_{R,t}^{*}$ and $D^{2}u^{*}_{R,t}$ on $\bar{B}_{R}\setminus B_{1}$, respectively.

 \begin{lem}\label{c1}
 Let $R>1$. There exists a constant $C>0$ depending on $R$, $a$ and $c$ such that for any $0\leq t\leq 1$ and any $u_{R,t}^{*}\in C^{2}(\bar{B}_{R}\setminus B_{1})$ satisfying \eqref{fd5a}-\eqref{fd7a}, we have that 
 \begin{equation*}
 |Du_{R,t}^{*}|\leq C~~~~\mbox{on }\bar{B}_{R}\setminus B_{1}.
 \end{equation*}
\end{lem}

 \begin{proof}
 By the strict convexity of $u^{*}_{R,t}$, we have that
 \begin{equation*}
 \max\limits_{\bar{B}_R\setminus B_{1}}|Du_{R,t}^{*}|=\max\limits_{\partial B_R\cup\partial B_{1}}|Du_{R,t}^{*}|.
 \end{equation*}
  By (\ref{fd6a}) and Lemma \ref{c0}, we have that
 \begin{equation*}
 |\partial u^{*}_{R,t}/\partial n|=|u^{*}_{R,t}-(1-t)\psi_t|\leq \max\limits_{\partial B_{1}}|\underline{\omega}_t|+\max\limits_{\partial B_{1}}|\bar{\omega}_t|+\max\limits_{\partial B_{1}}|\psi_t|~~~~\mbox{on }\partial B_{1}.
 \end{equation*}
 By the maximum principle, Lemma \ref{c0} and the above estimate, we have that  
 \begin{equation*}
 |Du^{*}_{R,t}|\leq |D\underline{\omega}_t|~~~~\mbox{on }\partial B_R.
 \end{equation*}

Now it only remains to estimate the derivatives of $u_{R,t}^{*}$ along any tangential vector field on $\partial B_{1}$. 
Suppose  that $T_{ij}:=p_i\frac{\partial}{\partial p_j}-p_j\frac{\partial}{\partial p_i}$ are the angular derivatives for $i\neq j$. Since any tangential vector on $\partial B_1$ can be expressed as a linear combination of $T_{ij}|_{\partial B_1}$ with bounded coefficients, we only need to estimate $T_{ij}u_{R,t}^*$.  Without loss of generality, we only consider $T:=T_{12}$. 

By Lemma  2.1 in \cite{ITW}, differentiating (\ref{fd5a}) with respect to $T$, we have that
 \begin{equation*}
F^{ij}(D^{2}u_{R,t}^{*})(Tu_{R,t}^{*})_{ij}=0~~~~\mbox{in }B_{R}\setminus\bar{B}_{1},
 \end{equation*}
where $F$ and $F^{ij}$ are defined as for any $d\times d$ real symmetric matrix $M=(M_{ij})$,
  \begin{eqnarray}\label{F}
  F(M):=\left(\frac{\sigma_{d}(\lambda(M)}{\sigma_l(\lambda(M))}\right)^{\frac{1}{d-l}} \text{    and    }F^{ij}(M):=\frac{\partial F}{\partial M_{ij}},~~i,j=1,\cdots,d.
  \end{eqnarray}
By the maximum principle, we have that
 \begin{equation*}
 |Tu_{R,t}^{*}|\leq \max\limits_{{\partial B_R}\cup\partial B_{1}}|Tu_{R,t}^{*}|~~~~\mbox{on }\bar{B}_{R}\setminus B_{1}.
 \end{equation*}
By \eqref{fd7a}, we have that
 \begin{equation}\label{chou1'}
 Tu_{R,t}^{*}=T\underline{\omega}_t~~~~\mbox{on }\partial B_{R}.
 \end{equation}
A straightforward calculation shows that for any $C^{2}$ function $v$, 
$$p\cdot DTv=T(p\cdot Dv).$$
Taking $v=u^*_{R,t}$, we have that 
\begin{eqnarray}\label{comm}
\frac{\partial }{\partial n}Tu^*_{R,t}=T\frac{\partial u_{R,t}^*}{\partial n}\ \ ~~~~\mbox{on }\partial B_{1}.
\end{eqnarray}

 If there exists $p_{0}\in\partial B_{1}$ such that $Tu_{R,t}^{*}(p_{0})=\max\limits_{\partial B_{R}\cup\partial B_{1}}Tu_{R,t}^{*}$, then we have that $(Tu_{R,t}^{*})_{n}(p_{0})\leq0$. It follows that
 \begin{eqnarray}\label{chou2}
 Tu_{R,t}^{*}(p_{0})&=&T((u_{R,t}^{*})_{n}+(1-t)\psi_t)(p_{0})\\
& =& (Tu_{R,t}^{*})_{n}(p_{0})+(1-t)T\psi_t(p_{0})\nonumber\\
&\leq& 2\max\limits_{\partial B_{1}}|D\psi_t|,\nonumber
 \end{eqnarray}
 where we have used \eqref{comm} in the second equality. 
 Analogously, if there exists $p_{0}\in\partial B_{1}$ such that $Tu_{R,t}^{*}(p_{0})=\min\limits_{\partial B_{R}\cup\partial B_{1}}Tu_{R,t}^{*}$, then we have that $(Tu_{R,t}^{*})_{n}(p_{0})\geq0$ and
 \begin{equation}\label{chou3}
 Tu_{R,t}^{*}(p_{0})\geq-2\max\limits_{\partial B_{1}}|D\psi_t|.
 \end{equation}
 Combining (\ref{chou1'}), (\ref{chou2}), (\ref{chou3}) together, we can obtain the desired estimate. 
 \end{proof}

  \begin{lem}\label{c2}
 Let $u^{*}_{R,t}$ be the solution of \eqref{fd5a}-\eqref{fd7a}. Then there exists some constant $C>0$ such that
 \begin{equation*}
 |D^{2}u_{R,t}^{*}|\leq C~~~~\mbox{on }\bar{B}_{R}\setminus B_{1},
 \end{equation*}
 when $R$ is sufficiently large. 
 \end{lem}

 \begin{proof}
 We will divide our proof into the following three steps.

 Step 1. We derive the double tangental derivatives on $\partial B_1$. In fact, we will prove that 
  \begin{equation}\label{chou12}
 |T^{2}u_{R,t}^{*}|\leq C~~~~\mbox{on }\bar{B}_{R}\setminus B_{1}.
 \end{equation}
 Indeed, by differentiating (\ref{fd5a}) with respect to $T$ twice, we have that
 \begin{equation*}
  \sum\limits_{i,j=1}^{d}F^{ij}(D^{2}u_{R,t}^{*})(T^{2}u_{R,t}^{*})_{ij}+\sum\limits_{i,j=1}^{d}\sum\limits_{p,q=1}^{d}F^{ij,pq}(D^{2}u_{R,t}^{*})(Tu_{R,t}^{*})_{ij}(Tu_{R,t}^{*})_{pq}=0,
 \end{equation*}
 where for any $d\times d $ real symmetric matrix $M$, $F(M)$, $F^{ij}(M)$ are defined as in \eqref{F} and 
 \begin{equation*}
 F^{ij,pq}(M):=\frac{\partial^{2}F}{\partial M_{ij}\partial M_{pq}}(M),~~~~~~~~i,j,p,q=1,\cdots,d.
 \end{equation*}
 Note that the calculation of the above equality can be found in Lemma 2.1 of \cite{ITW}. By the concavity of $F$, we have that
  \begin{equation*}
 \sum\limits_{i,j=1}^{d}F^{ij}(D^{2}u_{R,t}^{*})(T^{2}u_{R,t}^{*})_{ij}\geq0~~~~\mbox{in }B_{R}\setminus\bar{B}_{1}.
 \end{equation*}
 It follows from the maximum principle,
 \begin{equation*}
 T^{2}u_{R,t}^{*}\leq\max\limits_{\partial B_{R}\cup\partial B_{1}}T^{2}u_{R,t}^{*}~~~~\mbox{on }\bar{B}_{R}\setminus B_{1}.
 \end{equation*}
 By \eqref{fd7a}, we have that
 \begin{equation}\label{chou10}
 T^{2}u_{R,t}^{*}=T^2\underline{\omega}_t~~~~\mbox{on }\partial B_{R}.
 \end{equation}
If there exists $p_{0}\in\partial B_{1}$ such that $T^{2}u_{R,t}^{*}(p_{0})=\max\limits_{\partial B_{R}\cup\partial B_{1}}T^{2}u_{R,t}^{*}$, then we have that $(T^{2}u_{R,t}^{*})_{n}(p_{0})\leq0$. It follows that
 \begin{eqnarray}\label{chou11}
 T^{2}u_{R,t}^{*}(p_{0})&=&T^{2}((u_{R,t}^{*})_{n}+(1-t)\psi_t)(p_{0})\\
 &=&(T^{2}u_{R,t}^{*})_{n}(p_{0})+(1-t)T^{2}\psi_t(p_{0})\nonumber\\
 &\leq& 4\max\limits_{\partial B_{1}}(|D\psi_t|+|D^{2}\psi|),\nonumber
 \end{eqnarray}
 where we have used \eqref{comm} twice in the second equality. Combining (\ref{chou10}) and (\ref{chou11}) together, we can obtain (\ref{chou12}).

 Step 2. We derive the double normal derivatives on $\partial B_1$. Let
 \begin{equation*}
 \phi_{R,t}^{*}:=p\cdot Du_{R,t}^{*}-2u_{R,t}^{*}.
 \end{equation*}
 A straightforward calculation gives that
\begin{equation*}
(\phi_{R,t}^{*})_{ij}=\sum\limits_{l=1}^{d}p_{l}(u_{R,t}^{*})_{ijl}.
\end{equation*}
It follows that
\begin{equation*}
\sum\limits_{i,j=1}^{d}F^{ij}(D^{2}u_{R,t}^{*})(\phi_{R,t}^{*})_{ij}=\sum\limits_{l=1}^{d}p_{l}\sum\limits_{i,j=1}^{d}F^{ij}(D^{2}u_{R,t}^{*})(u_{R,t}^{*})_{ijl}=0~~~~\mbox{in }B_{R}\setminus\bar{B}_{1}.
\end{equation*}
By the maximum principle, we have that
\begin{equation*}
 \min\limits_{\partial B_{R}\cup\partial B_{1}}\phi_{R,t}^{*}\leq \phi_{R,t}^{*}\leq\max\limits_{\partial B_{R}\cup\partial B_{1}}\phi_{R,t}^{*}~~~~\mbox{on }\bar{B}_{R}\setminus B_{1}.
 \end{equation*}
 By \eqref{fd7a}, we have that, on $\partial B_R$, 
 \begin{eqnarray}\label{chou1}
 \phi_{R,t}^{*}\leq p\cdot D\underline{\omega}_t-2\underline{\omega}_t= 2c+O(|p|^{2-d^*_k(a(t))}),
 \end{eqnarray}
 for sufficiently large $R$. 
 On $\partial B_{1}$, we have 
 \begin{equation*}
\phi_{R,t}^{*}=\partial u^{*}_{R,t}/\partial n-2u^{*}_{R,t}=-u_{R,t}^{*}-(1-t)\psi_t\geq-u_{R,t}^{*}\geq-\bar{\omega}_t\geq-\max\limits_{\partial B_{1}}\bar{\omega}_t>2c,
 \end{equation*}
 by using of the condition $c< -\frac{a_i}{2}$ for $i=1,\cdots, d$. 
 Then for sufficiently large $R$, the maximum value of $\phi_{R,t}^{*}$ should be achieved on $\partial B_{1}$. It follows that $(\phi_{R,t}^{*})_{n}\leq0$ on $\partial B_1$, which implies that  
 $$(u^*_{R,t})_{nn}\leq 2 (u^*_{R,t})_n\leq C\mbox{  on $\partial B_1$}. $$ 

 Step 3. We finish the proof of this lemma.
Since we have 
 \begin{equation*}
 \sum\limits_{i,j=1}^{d}F^{ij}(D^{2}u_{R,t}^{*})(\Delta u_{R,t}^{*})_{ij}\geq0,
 \end{equation*}
 the maximum principle and the above two steps yield that
 \begin{equation*}
 \max\limits_{\bar{B}_{R}\setminus B_{1}}\Delta u_{R,t}^{*}=\max\limits_{\partial{B}_{R}\cup\partial B_{1}}\Delta u_{R,t}^{*}.
 \end{equation*}
By the convexity of $u^*_{R,t}$, Step 1 and Step 2, we have the $C^2$ estimate on $\partial B_1$. 

For the $C^2$ boundary estimate on $\partial B_R$, we would like to use the same argument as Guan in \cite{Guan}. 
However, our equation is a Hessian quotient equation, which does not satisfy the assumption (1.7) in \cite{Guan}. Therefore, in order to obtain the $C^2$ boundary estimates, we need to give a different proof of Lemma 6.2 in \cite{Guan}, i.e., we need to construct a barrier function  $\mathfrak{b}$ satisfying
\begin{eqnarray}\label{mmm}
\sum_{i,j}F^{ij}\mathfrak{b}_{ij}\leq -c(1+\sum_i F^{ii})\,\,\mbox{in $B_R\setminus B_1$}\mbox{ for some constant $c$},
\end{eqnarray}
and
\[\mathfrak{b}\geq 0\mbox {~in $B_R\setminus B_1$},\ \  \mathfrak{b}=0\,\,\mbox{on $ \partial B_R.$}\]
We define a new barrier function:  
$$\mathfrak{b}=u^*_{R,t}-\bar{\omega}_t+C(R^2-|p|^2).$$
It is easy to see that $\mathfrak{b}\geq 0$ in $B_R\setminus B_1$ and $\mathfrak{b}=0$ on $\partial B_R$.  Moreover, since 
$\sum_{i,j}F^{ij}(u^*_{R,t})_{ij}=F$, taking sufficiently large $C$, we can get \eqref{mmm}. 
We have obtained the desired estimate. 

 \end{proof}

We are in the position to solve  our Dirichlet-Neumann problem on the sufficiently large annulus.    
\begin{lem}\label{app}
For sufficiently large $R>1$, the following approximate problem on the annulus:
\begin{equation}\label{appprob}
\frac{\sigma_{d}(\lambda(D^{2}u^{*}_{R}))}{\sigma_{d-k}(\lambda(D^{2}u^{*}_{R}))}=\frac{1}{\binom{d}{k}}~~~~\mbox{in }B_{R}\texttt{\symbol{'134}}\bar{B}_{1}
\end{equation}
with the boundary conditions:
\begin{equation}\label{appbound}
u^{*}_{R}=\partial u^{*}_{R}/\partial n~~~~\mbox{on }\partial B_{1}\text{ and } u^{*}_{R}=\underline{\omega}_1~~~~\mbox{on }\partial B_{R}
\end{equation}
admits a unique strictly convex solution in $C^{2}(\bar{B}_{R}\setminus B_{1})$. 
\end{lem} 

\begin{proof}
the closedness follows from Lemma \ref{c0}-\ref{c2}. We only need to obtain the openness of the problem (\ref{fd5a})-(\ref{fd7a}).
We let 
\begin{eqnarray}
&&v=\frac{\partial u_{R,t}^{*}}{\partial t}, f=\frac{\partial}{\partial t} \left(\frac{\sigma_{d}(a(t))}{\sigma_{l}(a(t))}\right)^{1/(d-l)}, a_{ij}=F^{ij}(D^{2}u_{R,t}^{*}), \nonumber\\
&&\tilde{\psi}=-\psi_t+(1-t)\frac{\partial \psi_t}{\partial t}, \tilde{\varphi}=\frac{\partial\underline{\omega}_t}{\partial t}.\nonumber
\end{eqnarray}
Differentiating (\ref{fd5a})-(\ref{fd7a}) with respect to $t$, we have the linear problem on the annulus:   
\begin{equation}\label{fd5a''}
\sum\limits_{i,j=1}^{n}a_{ij}v_{ij}=f~~~~\mbox{in }B_{R}\texttt{\symbol{'134}}\bar{B}_{1},
\end{equation}
 with the boundary condition
\begin{equation}\label{fd6a''}
v=\frac{\partial v}{\partial n}+\tilde{\psi}~~~~\mbox{on }\partial B_{1}, \text{ and } v=\tilde{\varphi}~~~~\mbox{on }\partial B_{R}.\end{equation}

Since we do not find an appropriate reference for  the existence of the problem (\ref{fd5a''})-(\ref{fd6a''}), thus we give a sketch proof here.  We apply the continuity method to solve this problem. For any $0\leq s\leq 1$, we consider a family of the uniformly elliptic equations:
\begin{equation}\label{liuli}
L_{s}v:=s\sum_{i,j=1}^da_{ij}v_{ij}+(1-s)\Delta v=f~~~~\mbox{in }B_{R}\setminus\bar{B}_{1}
\end{equation}
with the boundary conditions \eqref{fd6a''}. Here $\Delta$ is the Laplace operator.

The closedness can be obtained as follows. By theorem 6.30 in \cite{gt}, for any solution $v\in C^{2,\alpha}(\bar{B}_{R}\setminus B_{1})$ of \eqref{liuli} and \eqref{fd6a''}, we have 
\begin{equation*}
\|v\|_{2,\alpha}\leq C(\|v\|_{0}+\|\tilde{\psi}\|_{1,\alpha}+\|\tilde{\varphi}\|_{2,\alpha}+\|f\|_{0,\alpha}).
\end{equation*}
Here $\|\cdot\|_{k,\alpha}$ is the $C^{k,\alpha}$ norm.  
The maximum principle also gives that  
\begin{equation*}
\|v\|_{0}\leq C(\|\tilde{\psi}\|_{0}+\|\tilde{\varphi}\|_{0}+\|f\|_{0}).
\end{equation*}
Thus, $\|v\|_{2,\alpha}$ is bounded by some constant independent of $s$. 
For the openness, we let $X_{1}:=C^{2,\alpha}(\bar{B}_{R}\setminus B_{1})$ and $X_{2}=C^{2,\alpha}(\partial B_{R})\times C^{1,\alpha}(\partial B_{1})$ with the norm $\|(\tilde{\varphi},\tilde{\psi})\|:=\|\tilde{\varphi}\|_{2,\alpha}+\|\tilde{\psi}\|_{1,\alpha}$. Then it is easy to see that the map $L_{s}:X_{1}\rightarrow X_{2}$ is  one-to-one and onto if and only if  the problem (\ref{liuli}), (\ref{fd6a''}) can be uniquely solved.

Thus, we only need to solve (\ref{liuli})-(\ref{fd6a''}) for  $s=0$, which is equivalent to
\begin{equation}\label{liuma}
\Delta u=f~~~~\mbox{in }B_{R}\setminus \bar{B}_{1}
\end{equation}
with the boundary conditions
\begin{equation}\label{liuma-}
u=\partial u/\partial n+\Phi~~~~\mbox{on }\partial B_{1} \text{ and } u=0~~~~\mbox{on }\partial B_{R}.
\end{equation}
We now using $L^2$ theory to solve the above problem. 
We say that $u$ is a weak solution of (\ref{liuma})-(\ref{liuma-}) if $u\in H_{0}^{1}(B_{R})$ satisfies
\begin{equation*}
B[u,\xi]:=\int_{B_{R}\setminus \bar{B}_{1}}D\xi\cdot Du+\int_{\partial B_{1}}\xi u=\int_{\partial B_{1}}\xi \Phi-\int_{B_{R}\setminus \bar{B}_{1}}\xi f.
\end{equation*}
Note that
\begin{equation*}
B[u,u]\geq \frac{\lambda}{2}\int_{B_{R}\setminus \bar{B}_{1}}|Du|^{2}-c\lambda\int_{B_{R}\setminus \bar{B}_{1}}u^{2}~~~~\mbox{for some constants $\lambda$ and $c$}.
\end{equation*}
Then the Lax-Milgram theorem will give the existence of the weak solution. Moreover, by the regularity results  in \cite{lie}, the weak solution is in fact the classical solution.

\end{proof}

\section{The Local Estimates}

For sufficiently large $R$, let $u^{*}_{R}$ be the solution of \eqref{appprob} and \eqref{appbound} obtained from Lemma \ref{app}. In this section, we will establish the local estimates of $u^{*}_{R}$ with respect to $R$ to obtain the existence of the solution of \eqref{fd5}-\eqref{fd7}.  

The local $C^0$ estimate can be stated as follows.

\begin{lem}\label{c01}
Let $R>R_{0}>1$. There exists a constant $C$ depending on $R_{0}$, $a$ and $c$ such that for any $u^{*}_{R}\in C^{2}(\bar{B}_{R}\setminus B_{1})$ satisfying \eqref{appprob} and \eqref{appbound}, we have that 
\begin{equation}\label{c0estimate---}
|u^*_R|\leq C   \text{ on }  \bar{B}_{R_0}\setminus B_{1}.
\end{equation}
\end{lem}

  \begin{proof}
  Let $\underline{\omega}_1$ and $\bar{\omega}_1$ be the sub and super solution constructed as in \eqref{w} and \eqref{sup}. Then, by the maximum principle, we have that 
   \begin{equation*}
 \underline{\omega}_1\leq u^{*}_{R}\leq \bar{\omega}_1~~~~\mbox{on }\bar{B}_{R}\setminus B_{1}.
 \end{equation*}
 (\ref{c0estimate---}) follows from the above inequality directly.
 \end{proof}

The local $C^1$ estimate can be stated as follows.
\begin{lem}\label{c11}
Suppose $R_0>0$ is a given sufficiently large constant.  For any $R>R_0+10$, 
 let $u^{*}_{R}$ be the solution of \eqref{appprob} and \eqref{appbound}. Then there exists a constant $C>0$ depends only on $R_0$ such that
 \begin{equation*}
 |Du_{R}^{*}|\leq C   \text{ on }  \bar{B}_{R_0}\setminus B_{1}.
 \end{equation*}
 \end{lem}

 \begin{proof}
 Let $r=|p|$ be the radius of the points. By the convexity, we have $\frac{\partial^2 u^*_{R}}{\partial r^2}>0$. Therefore, the minimum value of  $\frac{\partial  u^*_{R}}{\partial r}$ can be achieved on $\partial B_1$. Thus $\frac{\partial  u^*_{R}}{\partial r}$ has a uniform lower bound in $B_R\setminus B_1$ by Lemma \ref{c01}.  Suppose that the set $\left\{ p\in B_{R_0+10}\setminus \bar{B}_1:\frac{\partial u^*_{R}}{\partial r} (p)>0\right\}$ is non empty. We consider the test function
$$\phi=((R_0+10)^2-|p|^2) \frac{\partial u^*_{R}}{\partial r} e^{u^*_{R}}. $$
Assume  $\phi$ achieves its maximum value at some point $p_0$. If $p_0\in \partial B_1$, by \eqref{appbound} and Lemma \ref{c01}, we get an upper  bound of $\phi$. If $p_0\in B_{R_0+10}\setminus\bar{B}_1$, then at $p_0$, we have $\frac{\partial \phi}{\partial r}=0$, which gives
$$-2r\frac{\partial u^*_{R}}{\partial r}+((R_0+10)^2-|p|^2) \frac{\partial^2 u^*_{R}}{\partial r^2}+((R_0+10)^2-|p|^2) \left(\frac{\partial u^*_{R}}{\partial r}\right)^2=0.$$ Thus, we get
$$((R_0+10)^2-|p|^2)\frac{\partial u^*_{R}}{\partial r}\leq 2r\leq 2R_0+20.$$ It gives an upper bound of $\phi$. Thus, we get the uniform upper and lower bound of  $\frac{\partial u^*_{R}}{\partial r}$ on $\bar{B}_{R_0+9}\setminus B_1$.

Let us estimate the angular derivatives. We consider the function
$$\psi(p,v)=((R_0+5)^2-|p|^2)v\cdot Du^*_{R}(p) e^{\epsilon u^*_{R}(p)}.$$
where $v$ is some vector and $\epsilon>0$ is a small undetermined  constant.
Define the set
$$S=\{(p,v); p\in B_{R_0+5}\setminus \bar{B}_1, v \in T_{p}B_{|p|}, \text{ and } |v|=|p|  \}$$
Here $B_{|p|}$ denotes the ball centered at origin with radius $|p|$ and $T_{p}B_{|p|}$ is its tangential space at $p$. Assume $\psi(p,v)$ achieves its maximum value at point $(p_0,v_0)$ on the closure of $S$.
We rotate the coordinate to satisfy $p_0=(|p_0|,0,\cdots,0)$ and $p_2$ parallel to $v_0$. Thus, denote $T=p_1\frac{\partial}{\partial p_2}-p_2\frac{\partial }{\partial p_1}$, then we have the test function
$$\psi_1=((R_0+5)^2-|p|^2)Tu^*_{R} e^{\epsilon u^*_{R}}$$
also achieves its maximum value at $p_0$ and $\psi_1(p_0)=\psi(p_0,v_0)$.
If $p_0\in \partial B_1$, we have $\frac{\partial \psi_1}{\partial n}\leq 0$, which gives
\begin{eqnarray*}
0&\geq& -2Tu^*_{R}+((R_0+5)^2-1)\frac{\partial Tu^*_{R}}{\partial n}+((R_0+5)^2-1)\epsilon Tu^*_{R}\frac{\partial u^*_{R}}{\partial n}\\
&=&((R_0+5)^2-3) Tu^*_{R} +((R_0+5)^2-1)\epsilon u^*_{R} Tu^*_{R}.
\end{eqnarray*}
Since by Lemma \ref{c01}, we have the uniform lower bound of $u^*_{R}$. If $\epsilon$ is sufficiently small, the above inequality gives a contradiction. Thus, $p_0\in B_{R_0+5}\setminus \bar{B}_1$. At $p_0$,  we have  $\frac{\partial \psi_1}{\partial p_2}=0$, which implies
$$0=((R_0+5)^2-|p|^2)\left(|p_0|\frac{\partial^2u^*_{R}}{\partial p_2^2}-\frac{\partial u^*_{R}}{\partial p_1}\right)+((R_0+5)^2-|p|^2)\frac{\epsilon}{|p_0|} (Tu^*_{R})^2.$$
Thus, we get
$$Tu^*_{R}\leq C\sqrt{\frac{\partial u^*_{R}}{\partial r}}.$$
Therefore, we get an uniform upper bound of $\psi_1$, which implies an uniform upper bound of the angular derivatives of $u^*_{R}$ in $B_{R_0}\setminus B_1$. Thus, we get the uniform bound of $|Du^*_{R}|$ in $B_{R_0}\setminus B_1$.

 \end{proof}
 
 Let us consider the  local $C^2$ estimate. We give the asymptotic behavior of $D^2u^*_R$ at first, then a similar argument as Lemma \ref{c2} will give the local $C^2$ estimate.  However,  the difficulty is lack of the Pogorelov type estimate for Hessian quotient equations, if  we would like to generalize the argument of Caffarelli-Li \cite{cl2003} for Monge-Amp\`{e}re equations.   To overcome the difficulty, our strategy is to consider the asymptotic behavior of $D^2u_R$, instead of $D^2u^*_R$, where $u_R$ is Legendre transform of $u^*_R$.

 By Lemma \ref{c01} and the asymptotic behaviors of $\bar{\omega}_1$ and $\underline{\omega}_1$, there exist two sufficiently large positive constants $r_{0}, C_{0}$ depending on $d$ and $a$, such that for any $R>r_{0}$, we have that
 \begin{equation}\label{fnfn}
 \left|u_{R}^{*}-\frac{1}{2}\sum\limits_{i=1}^{d}a_{i}p_{i}^{2}+c\right|\leq C_{0}|p|^{2-d^*_k(a)}~~~~\mbox{on }\bar{B}_{R}\setminus B_{r_{0}}.
 \end{equation}
 Let $u_R$ be the Legendre transform of $u^*_R$ and $\Omega_R=Du_{R}^{*}(B_R\setminus \bar{B}_1)$. Since $Du^*_R$ is a diffeomorphism, we can denote $\Gamma_1, \Gamma_2$ to be the exterior boundary and interior boundary of $\Omega_R$.   
Then $u_{R}$ satisfies that
 \begin{equation*}
 \sigma_{k}(\lambda(D^{2}u_{R}))=\binom{d}{k}~~~~\mbox{in }\Omega_R.
 \end{equation*}
 Then, we have
 \begin{lem}\label{gradient}
There exist two uniform constants $\theta,\tilde{C}$ such that  
 \begin{equation}\label{bound}
 \Gamma_1\subset \mathbb{R}^d\setminus \bar{B}_{\theta R},\text{ and } \ \ \Gamma_2\subset B_{\tilde{C}}.
 \end{equation}
 Moreover, for sufficiently large $R$, on the annulus $\bar{B}_{\theta R}\setminus B_{\tilde{C}}\subset \Omega_R$, we have
 \begin{equation}\label{fnfn'}
\left |u_{R}-\frac{1}{2}\sum\limits_{i=1}^{d}\frac{1}{a_{i}}x_{i}^{2}-c\right|\leq C_{0}|x|^{2-d^*_k(a)}. 
\end{equation}
\end{lem}
\begin{proof}
For sufficiently large $R$,
on the boundary $\partial B_R$, by the convexity, it is clear that 
\begin{eqnarray}
\frac{\partial u^*_R}{\partial n}&\geq& \frac{\min\limits_{\partial B_R}u^*_R-\max\limits_{\partial B_1}u^*_{R}}{R-1}\geq\frac{\min\limits_{\partial B_R}\underline{\omega}_1-\max\limits_{\partial B_1}\bar{\omega}_1}{R-1}\\
&\geq&\frac{1}{R-1}\left[\frac{1}{2}\sum_{i=1}^da_ip_i^2-c-C\right]\nonumber\\
&\geq&\frac{1}{R-1}\left[\frac{R^2}{2}\min_{1\leq i\leq d} a_i-c-C\right]\nonumber\\
&\geq& \theta R,\nonumber
\end{eqnarray}
where $\theta$ is a small positive constant. Therefore, $|Du^*_R|\geq \theta R$ on $\partial B_R$, which implies the first result of \eqref{bound}. The second result of \eqref{bound} is a corollary of Lemma \ref{c11} if we take $R_0=2$ there. 

It is obvious that $$\bar{\omega}^*_1(x)=x\cdot p-\bar{\omega}_1(p)=\frac{1}{2}\sum_{i=1}^d\frac{x_i^2}{a_i}+c,$$ and it is defined on the whole $\mathbb{R}^n$. Then, let's consider $\underline{\omega}^*_1$, the Legendre transform of $\underline{\omega}_1$. By \eqref{w}, $\underline{\omega}_1$ is defined on $r=\sqrt{\sum_ia_ip_i^2}\geq \eta$. Then we have 
$$x_i=D_i\underline{\omega}_1=a_ip_i\xi^{-1}\left(C_1r^{-d^*_k(a)}\right)=a_ip_ih(r),$$
which implies
\begin{equation}\label{px}
\sum_{i=1}^d\frac{x_i^2}{a_i}=r^2h^2(r).
\end{equation}
Here $h(r)=\xi^{-1}\left(C_1r^{-d^*_k(a)}\right)$. Furthermore, the derivative of the function $r^2h(r)$ is
$$2rh(r)+r^2h'(r)\geq rh(r)>0,$$
where we have used \eqref{shengguo}. Thus, the range of the function $r^2h(r)$ is $[\eta^2h(\eta),\infty)$ in view of $h\geq 1$. 
 Therefore, we known that
 $$D\underline{\omega}_1= \mathbb{R}^d\setminus \left\{x\in\mathbb{R}^n; \sum_i\frac{x_i^2}{a_i}\leq \eta^2h(\eta)\right\}.$$  
Thus, $\underline{\omega}_1^*$ can be defined on $\mathbb{R}^d\setminus B_{\tilde{C}}$, for $\tilde{C}>(\min_{i} a_i) \eta^2h(\eta)$. In view of \eqref{px}, we know that, $$r^2\leq \sum_{i=1}^d\frac{x_i^2}{a_i}\leq r^2h^2(\eta).$$ Combining with \eqref{asyw}, we get
$$\underline{\omega}_1^*=\frac{1}{2}\sum_{i=1}^d\frac{x_i^2}{a_i}+c+O(|x|^{2-d^*_k(a)}), \text{ as } |x|\rightarrow+\infty.$$
At last, we prove, on $\bar{B}_{\theta R}\setminus B_{\tilde{C}}$, 
$$\bar{\omega}_1^*\leq u_R\leq \underline{\omega}_1^*.$$
We only prove the first inequality, and the second one is same. For any $x=D\bar{\omega}_1(p_1)=Du^*_R(p_2)\in\bar{B}_{\theta R}\setminus B_{\tilde{C}}$, we have 
\begin{eqnarray}
\bar{\omega}_1^*(x)-u_R(x)&=&x\cdot p_1-\bar{\omega}_1(p_1)-x\cdot p_2+u^*_R(p_2)\nonumber\\
&\leq& x\cdot (p_1-p_2)-u^*_R(p_1)+u^*_R(p_2)\leq 0\nonumber. 
\end{eqnarray}
Here the last inequality  comes from the strict convexity of $u^*_R$. Thus, by the asymptotic behavior of $\underline{\omega}_1^*, \bar{\omega}_1^*$, we obtain  \eqref{fnfn'}. 
\end{proof}

Now,  we can derive the asymptotic behavior of  $Du_{R}^{*}$ and $D^2u^*_R$.

 \begin{lem}\label{cl}
 There exist two sufficiently large constants $r_{1}>1$ and $C_{1}>0$ such that for any $\theta R>2r_1$ and $i$, $j=1$, $\cdots$, $d$, we have
  \begin{equation}\label{clc}
 |D_{i}u_{R}^{*}-a_{i}p_{i}|\leq C_{1}|p|^{1-d_k^*(a)} \text{ and }\ \  |D_{ij}u_{R}^{*}-a_{i}\delta_{ij}|\leq C_{1}|p|^{-d_k^*(a)}
 \end{equation}
 for $r_{1}\leq|p|<\theta R/2$.  Here $\theta$ is the constant given in Lemma \ref{gradient}. 
 
 \end{lem}

 \begin{proof} We follow the proof of Lemma 3.3 in Bao-Wang \cite{BaoW}. For $s>0$, define
$$D_s=\left\{y\in\mathbb{R}^d: \frac{1}{2}\sum_{i=1}^d\frac{y_i^2}{a_i}<s\right\}.$$ 
For sufficiently large $R$, by Lemma \ref{gradient}, $u_R$ can be defined on  the annulus $\bar{B}_{\theta R}\setminus B_{\tilde{C}}$.
For $x\in \mathbb{R}^d$, $|x|\in[2\tilde{C},\frac{2}{3}\theta R]$, we let $r=(\max_i a_i)^{-1/2}|x|$ and 
\begin{equation*}
\eta_{r}(y):=\frac{16}{r^2}u_{R}\left(x+\frac{r}{4}y\right),~~~~y\in D_2.
\end{equation*}
It is easy to see that  $\eta_{r}$ satisfies
\begin{equation*}
\sigma_{k}(\lambda((D^{2}\eta_{r}))=\binom{d}{k}, ~~D^{2}\eta_{r}>0,~~~~\mbox{in }D_{2}.
\end{equation*}
Let
\begin{equation*}
w_{r}(y):=\eta_{r}(y)-\frac{8}{r^2}\sum\limits_{i=1}^{n}\frac{1}{a_i}\left(x_{i}+\frac{r}{4}y_{i}\right)^{2}-\frac{16}{r^2}c,\end{equation*}
and  $$\bar{\eta}_r(y)=w_r(y)+\frac{1}{2}\sum_{i=1}^d\frac{y^2_i}{a_i}.$$
It is clear that 
\begin{equation*}
\sigma_{k}(\lambda((D^{2}\bar{\eta}_{r}))=\binom{d}{k}.
\end{equation*}

For $M\leq 2$, let
\begin{equation*}
\Omega_{M,r}:=\{y\in D_2; \bar{\eta}_{r}(y)<M\}.
\end{equation*}
By (\ref{fnfn'}), we have, for $y\in D_2$,
$$\left|\bar{\eta}_r(y)-\frac{1}{2}\sum_{i=1}^d\frac{y^2_i}{a_i}\right|=|w_r(y)|\leq \frac{C}{r^2}\left(x+\frac{r}{4}y\right)^{2-d^*_k(a)}\leq Cr^{-d^*_k(a)}$$
and $|\eta_r|\leq C$. Thus, for sufficiently large $r_1$, if $r>r_1$, $\Omega_{1.5,r}\subset D_{1.6}$. Applying the interior gradient estimate in \cite{cw2001} to $\bar{\eta}_r$ in $D_2$, we have that
\begin{equation*}
\|D\bar{\eta}_{r}\|_{L^{\infty}(D_{1.6})}\leq C.
\end{equation*}
Applying the Pogorelov type second order derivative estimate in \cite{cw2001} to $\bar{\eta}_r$ in $\Omega_{1.5,r}$, we have that
\begin{equation*}\label{dn}
(\bar{\eta}(y)-1.5)^4|D^{2}\bar{\eta}_{r}(y)|\leq C, 
\end{equation*}
which implies 
$$\|D^2\bar{\eta}_r\|_{L^{\infty}(\Omega_{1.2,r})}\leq C.$$
Then $w_{r}$ satisfies that
\begin{equation*}
\sum_{i,j=1}^d a^{ij}(w_{r})_{ij}=0,~~~~\mbox{in }D_{2},
\end{equation*}
where $$a^{ij}=\int^1_0\sigma_k^{ij}(\lambda(A+sD^2w_r))ds.$$
Since $A+sD^2w_r=(1-s)A+sD^2\bar{\eta}$ and $(\sigma_k^{ij}(\lambda(A))),(\sigma_k^{ij}(\lambda(D^2\bar{\eta}_r)))$ has uniformly lower and upper bound, therefore $(a_{ij})$ also has uniformly lower and upper bound,
\begin{equation*}
\frac{I}{C}<(a_{ij})<CI,~~~~\mbox{in }D_{1}.
\end{equation*}
By the Schauder theory, we have 
\begin{equation*}
|D^{m}w_{r}(0)|\leq C\|w_{r}\|_{L^{\infty}(D_{1})}\leq Cr^{-d^*_k(a)},
\end{equation*}
which implies that
\begin{equation}\label{uRasy}
\left|D^{m}\left(u_{R}-\frac{1}{2}\sum\limits_{i=1}^{n}\frac{1}{a_{i}}x_{i}^{2}-c\right)\right|\leq C|x|^{2-d^*_k(a)-m},
\end{equation}
for  $(\max_i a_i)^{1/2}r_1\leq |x|\leq \frac{2}{3} \theta R$. 

By \eqref{uRasy}, for $m=1$, we have
$$p_i=D_iu_R=\frac{x_i}{a_i}+O(|x|^{1-d^*_k(a)}),$$ for sufficiently large $|x|$, which implies the first inequality of \eqref{clc}, if  $r_1$ is sufficiently large.  For $m=2$, since we known $(D^2u_R)=(D^2u^*_R)^{-1}$, then we have 
$$|(D^2u^*_R)^{-1}-A|\leq C|p|^{-d^*_k(a)}.$$ 
This gives the bound of $D^2u^*_R$, 
which implies the second inequality of \eqref{clc} for sufficiently large $r_1$.    

 \end{proof}

We are in the position to give the local $C^2$ estimate. 

\begin{lem}\label{localC2}
Suppose $R_0>0$ is a given sufficiently large constant.  For any $\frac{\theta R}{2}>R_0+10$, 
 let $u^{*}_{R}$ be the solution of (\ref{appprob}) and (\ref{appbound}), where $\theta$ is the constant given in Lemma \ref{gradient}. Then there exists a constant $C$ only depending on $R_0$ such that
 \begin{equation*}
 |D^2u_{R}^{*}|\leq C   \text{ on }  \bar{B}_{R_0}\setminus B_{1}
 \end{equation*}
 \end{lem}
\begin{proof}
Using Lemma \ref{cl}, we only need the $C^2$ estimate on $\bar{B}_{r_1}\setminus B_1$. We use the same argument as in Lemma \ref{c2}. 
The proof of Step 1 and Step 3 are the same as in Lemma \ref{c2}, therefore, we only need to reprove Step 2. 

Let
 \begin{equation*}
 \phi_{R}^{*}:=p\cdot Du_{R}^{*}-2u_{R}^{*}.
 \end{equation*}
 Same argument  as Step 2 shows 
 \begin{equation*}
  \phi_{R}^{*}\leq\max\limits_{\partial B_{r_1}\cup\partial B_{1}}\phi_{R}^{*}~~~~\mbox{on }\bar{B}_{r_1}\setminus B_{1}.
 \end{equation*}
 By \eqref{fnfn'} and \eqref{clc}, we have that, on $\partial B_{r_1}$, 
 \begin{eqnarray}\label{chou1}
 \phi_{R}^{*}= 2c+O(|p|^{2-d^*_k(a)}),\nonumber
 \end{eqnarray}
 for sufficiently large $r_1$. 
 On $\partial B_{1}$, we still have 
 \begin{equation*}
\phi_{R}^{*}=\partial u^{*}_{R}/\partial n-2u^{*}_{R}=-u_{R}^{*}\geq-\bar{\omega}_1\geq-\max\limits_{\partial B_{1}}\bar{\omega}_1>2c,
 \end{equation*}
 by using of the condition $c< -\frac{a_i}{2}$ for $i=1,\cdots, d$. 
 Then for sufficiently large $r_1$, the maximum of $\phi_{R}^{*}$ should achieve on $\partial B_{1}$. It follows that $(\phi_{R}^{*})_{n}\leq0$, which implies, on $\partial B_1$,
 $$(u^*_{R})_{nn}\leq 2 (u^*_{R})_n\leq C. $$ 
\end{proof}

Combining local $C^0-C^2$ estimates with asymptotic behavior \eqref{fnfn'}, using Cantor's subsequence method, we obtain a unique solution $u^*$ to the dual problem  \eqref{fd5}-\eqref{fd7}. The uniqueness comes from the maximum principle. Suppose $u$ is the Legendre transform of $u^*$. Then, $u$ should satisfy $\sigma_k(\lambda(D^2u))=\binom{d}{k}$.  Let $D=Du^*(\mathbb{R}^d\setminus \bar{B}_1)$, which is an open set.  
Let $u^*_R=u^*|_{B_R}$ for sufficiently large $R$. Then using the same argument, we also can prove that Lemma \ref{gradient} holds for $u^*_R$, which is 
$$Du^*_R(\partial B_R)\subset \mathbb{R}^d\setminus \bar{B}_{\theta R}, \text{ and } Du^*_R(\partial B_1)\subset B_{\tilde{C}}.$$
Here $\theta, \tilde{C}$ are two constants not depending on $R$. Since $R$ can go to infinity, we have $\mathbb{R}^d\setminus \bar{B}_{\tilde{C}}\subset D$. Moreover, $Du^*(\partial B_1)$ is the interior boundary of $D$. Let $\Omega$ be the domain enclosed by $Du^*(\partial B_1)$. Then, we have $\partial\Omega=Du^*(\partial B_1)$ and $D=\mathbb{R}^d\setminus \bar{\Omega}$. By \eqref{fd6}, $u=0$ on $\partial\Omega$. Since $u$ is a strictly convex function, we get the strict convexity of $\Omega$. Again using Lemma \ref{gradient}, we further have the asymptotic behavior of $u$ by \eqref{fnfn'}. Moreover, by Lemma 3.3 in \cite{BaoW}, we can conclude that the decay rate for $u$ approaching the quadratic polynomial is \eqref{fd4}. At last, let's prove $\frac{\partial u}{\partial n}=1$ on $\partial \Omega$. Since 
on $\partial \Omega$, $1=|Du|=|u_n|$, then $u_n=\pm 1$. Since $u=0$ on $\partial\Omega$, if $u_n=-1$, one can see $u<0$ near $\partial\Omega$. By the asymptotic behavior \eqref{fd4'}, $u>0$ at infinity. Thus, $u$ can achieve negative minimum value at some point $P$ in $D$. Thus, we have $Du(P)=0$ which is a contradiction for $Du(D)=\mathbb{R}^d\setminus \bar{B}_1$.   Therefore, we get
$u_n=1$ on $\partial \Omega$.  Thus, we have proved $u$ satisfying \eqref{baozhang3}, \eqref{baozhang2} and \eqref{fd4'}.

We complete the proof of Theorem \ref{cases:10}, \ref{cases:20} and Remark \ref{sm}.

\section{Radially symmetric solutions}

We firstly give the 

\begin{proof}[proof of Theorem \ref{cases:80}]

Without loss of generality, we may assume that $b=0$.

For the existence and non-existence part, we can assume that $\Omega$ is a ball centered at the origin with radius $r_{0}>0$. We will construct the radially symmetric solution $u(x)=u(r)$ with $r=|x|$ of the problem (\ref{baozhang3})-(\ref{ab}).

A straightforward calculation shows that $h=u'/r$ satisfies the ODE
\begin{equation*}\label{ODE}
h^k+\frac{k}{d}rh'h^{k-1}=1,~~~~\forall~r>r_{0}.
\end{equation*} 
Then the solution of the above equation with boundary conditions $u(r_{0})=0$ and $u'(r_{0})=1$ is
\begin{equation}\label{app}
u(r)=\int_{r_0}^{r}s\left(1+Cs^{-d}\right)^{1/k}ds,~~~~\forall~r\geq r_{0},
\end{equation}
where $C=C(r_{0})=r_0^{d-k}-r_0^d$. Moreover, we have that 
\begin{equation}\label{app2}
u(r)=\frac{1}{2}r^{2}+\mu(r_0)+O(r^{2-d}),~~~~\mbox{as $r\rightarrow\infty$}, 
\end{equation}
where 
\begin{equation}\label{hailizi}
\mu(r_0)=-\frac{1}{2}r_{0}^{2}+\int_{r_0}^{\infty}s\left(\left(1+Cs^{-d}\right)^{1/k}-1\right)ds.
\end{equation}

The strictly convexity of $u$ gives $u''>0$, which implies that 
$$r^d-\frac{d-k}{k}C>0,~~~~\forall~r\geq r_0.$$
Thus, we only need $r^d_{0}-\frac{d-k}{k}C>0$, which gives $r_0>r_{2}:=\left(\frac{d-k}{d}\right)^{1/k}$. 

Now we discuss the range of $\mu(r_0)$ for $r_0>r_{2}$. It is obvious that $\mu\rightarrow -\infty$ as $r_0\rightarrow \infty$. The derivative of $\mu$ is 
$$\mu'(r_0)=-1+\frac{d}{k}r_0^{d-1}(r_{2}^{k}r_0^{-k}-1)\int^{\infty}_{r_0}s^{1-d}(1+Cs^{-d})^{1/k-1}ds.$$
For $k=d$,  it is clear that $r_2=0$ and $\mu'<0$. Thus, $$\mu(r_0)< \mu(0)=\int_{0}^{\infty}s((1+s^{-d})^{1/d}-1)ds.$$
For $1\leq k\leq d-1$, by the change of variable $s=r_0\tau$, $\mu$ and $\mu'$ can be written as
$$\mu(r_0)=r_0^2\left[-\frac{1}{2}+\int^{\infty}_1\tau\left((1+(r_0^{-k}-1)\tau^{-d})^{1/k}-1\right)d\tau\right],$$
and
$$\mu'(r_0)=-1+\frac{d}{k}r_0(r_2^k r_0^{-k}-1)\int^{\infty}_{1}\tau^{1-d}(1+(r_0^{-k}-1)\tau^{-d})^{1/k-1}d\tau.$$
Since $r_2<1$, we have that $\mu'(r_{0})<0$ for any $r_{0}>r_{2}$. Thus, $$\mu(r_0)< \mu(r_2)=\left(\frac{d-k}{d}\right)^{2/k}\left[-\frac{1}{2}+\int^{\infty}_1\tau\left(\left(1+\frac{k}{d-k}\tau^{-d}\right)^{1/k}-1\right)d\tau\right].$$

For the uniqueness part, if ($\Omega_{1}$, $u_{1}$) and ($\Omega_{2}$, $u_{2}$) are two pairs of domain and solution such that the problem (\ref{baozhang3})-(\ref{ab}) admits a unique strictly convex solution, then $u^{*}_{1}$ and $u_{2}^{*}$, the Legendre transform of $u_{1}$ and $u_{2}$, should be the same by the maximum principle. Thus, we obtain that $\Omega_{1}=\Omega_{2}$ and $u_{1}\equiv u_{2}$. 

\end{proof}

We secondly give the 
\begin{proof}[proof of Remark \ref{pangxie}]

Without loss of generality, we may assume that $b=0$. 

Suppose that $\Omega$ is a ball centered at the origin with radius $r_{0}>0$. We will construct the radially symmetric solution $u(x)=u(r)$ with $r=|x|$ of the problem (\ref{baozhang3})-(\ref{ab}). A straightforward calculation shows that $u$ is of the form (\ref{app}). Moreover, $u$ satisfies the asymptotic behavior (\ref{app2}) with $\mu(r_{0})$ defined as in (\ref{hailizi}).

Now we need to derive the range of $\mu(r_{0})$ for any $r_{0}>0$.  For $k=1$, we have  that
\begin{equation*}
\mu(r_{0})=-\frac{1}{2(d-2)}(r_{0}-1)^{2}+\frac{1}{2(d-2)}.
\end{equation*}
Then it is easy to see that 
\begin{equation*}
\mu(r_{0})\leq\hat{c}(1,d):=\max\limits_{[0,\infty)}\mu=\mu(1)=\frac{1}{2(d-2)}.
\end{equation*}
For $2\leq k\leq d-1$, since $\lim\limits_{r_{0}\rightarrow\infty}\mu(r_{0})=-\infty$, $\mu(0)=0$ and $\liminf\limits_{r_{0}\rightarrow0^{+}}\mu'(r_{0})>0$, there exists $r_{3}>0$ such that 
\begin{equation*}
\mu(r_{0})\leq \hat{c}(k,d):=\max\limits_{[0,\infty)}\mu=\mu(r_{3}).
\end{equation*}
Note that the solutions with different $k$ have the same asymptotic behavior. 
Then the monotonicity of $\hat{c}$ with respect to $k$ follows directly from the monotonicity of $h$ with respect to $k$. 
\end{proof}

\section{Some ridigity results}

We firstly give the 
\begin{proof}[proof of Theorem \ref{ridig''}]

We take $R>0$ sufficiently large such that $\bar{\Omega}\subset B_{R}$. The asymptotic behavior (\ref{fd4--}) implies that as $|x|\rightarrow\infty$,
\begin{equation}\label{chouba}
Du(x)=x+O(|x|^{-(\alpha+1)}),~~~~D^{2}u=I+O(|x|^{-(\alpha+2)}).
\end{equation}
It follows that as $|x|\rightarrow\infty$, 
\begin{equation}\label{choudan}
\sigma_{k}^{ij}:=\sigma_{k}^{ij}(\lambda(D^{2}u))=
\begin{cases}
\binom{d-1}{k-1}\delta_{ij}+O(|x|^{-(\alpha+2)}),&i=j,\\
O(|x|^{-(\alpha+2)}),&i\neq j.
\end{cases}
\end{equation}
for $i$, $j=1$, $\cdots$, $d$.

Then for sufficiently large $R>0$, we can compute the following boundary terms:
\begin{equation*}
\int_{\partial B_{R}}\sigma_{k}^{ij}u_{j}x_{i}=k\binom{d}{k}R|B_{R}|+O(R^{d-1-\alpha}),
\end{equation*}
\begin{equation*}
\int_{\partial B_{R}}(x\cdot Du)\sigma_{k}^{ij}u_{j}x_{i}=k\binom{d}{k}R^{3}|B_{R}|+O(R^{d+1-\alpha}),
\end{equation*}
\begin{equation*}
\int_{\partial B_{R}}\sigma_{k}^{ij}u_{i}u_{j}=k\binom{d}{k}R|B_{R}|+O(R^{d-1-\alpha}),
\end{equation*}
\begin{equation*}
\int_{\partial B_{R}}x_{l}u\frac{\partial \sigma_{k}^{ij}}{\partial x_{l}}u_{j}x_{i}=O(R^{d+1-\alpha}),
\end{equation*}
\begin{equation*}
\int_{\partial B_{R}}u=\frac{1}{2}dR|B_{R}|+cdR^{-1}|B_{R}|+O(R^{d-1-\alpha}),
\end{equation*}
\begin{equation*}
\int_{\partial B_{R}}u\sigma_{k}^{ij}u_{j}x_{i}=\frac{1}{2}k\binom{d}{k}R^{3}|B_{R}|+k\binom{d}{k}cR|B_{R}|+O(R^{d+1-\alpha}).
\end{equation*}

Integrating the equation (\ref{baozhang3}) on $B_{R}\setminus\bar{\Omega}$, we have that 
\begin{equation}\label{yaduo}
\int_{B_{R}\setminus\bar{\Omega}}\sigma_{k}(\lambda(D^{2}u))=\binom{d}{k}(|B_{R}|-|\Omega|).
\end{equation}
The left hand side can be computed as 
\begin{align*}
\int_{B_{R}\setminus\bar{\Omega}}\sigma_{k}(\lambda(D^{2}u))&=\frac{1}{k} \int_{B_{R}\setminus\bar{\Omega}}(\sigma_{k}^{ij}u_{j})_{i}=\frac{1}{kR}\int_{\partial B_{R}}\sigma_{k}^{ij}u_{j}x_{i}-\frac{1}{k}\int_{\partial\Omega}\sigma_{k}^{ij}u_{i}u_{j}\\
&=\binom{d}{k}|B_{R}|-\frac{1}{k}\int_{\partial\Omega}\sigma_{k}^{ij}u_{i}u_{j}+O(R^{d-2-\alpha}).
\end{align*}
Inserting the above equality into (\ref{yaduo}) and letting $R\rightarrow\infty$, we have that 
\begin{equation*}
|\Omega|=\frac{1}{k\binom{d}{k}}\int_{\partial\Omega}\sigma_{k}^{ij}u_{i}u_{j}=\frac{1}{k\binom{d}{k}}\int_{\partial\Omega}H_{k-1}.
\end{equation*}
Here $H_{k}$ denotes the $k$-th curvature of the level set of $u$. And it is well-known that 
\begin{equation*}
\sigma_{k}(\lambda(D^{2}u))=H_{k}|Du|^{k}+\frac{\sigma_{k}^{ij}u_{i}u_{lj}u_{l}}{|Du|^{2}},~~~~H_{k-1}=\frac{\sigma_{k}^{ij}u_{i}u_{j}}{|Du|^{k+1}}.
\end{equation*}

We first prove that $x\cdot Du\geq 0$ and $u\geq 0$. We define a function 
$$\varphi=x\cdot Du-2u.$$
It is obvious that $\varphi\rightarrow -2c > 0$ and $\varphi\geq 0$ on $\partial\Omega$. Moreover, one can prove 
$$\sigma_k^{ij}\varphi_{ij}=0.$$ Thus,  we get $\varphi\geq 0$ in $\mathbb{R}^d\setminus \Omega$. Therefore, we get
$$\frac{1}{r^2}\left(x\cdot Du-2u\right)=x\cdot D(\frac{u}{r^2})\geq 0,$$ which implies 
$$\frac{u}{r^2}\geq \left.\frac{u}{r^2}\right|_{\partial \Omega}=0.$$ Thus, we get $u\geq 0$ and $x\cdot Du\geq 2u\geq 0$.  

By the Newton's inequality and (\ref{baozhang3}), we have that 
\begin{equation}\label{shenxia}
\Delta u\geq d\left(\frac{\sigma_{k}(\lambda(D^{2}u))}{\binom{d}{k}}\right)^{1/k}=d.
\end{equation}

Multiplying inequality (\ref{shenxia}) with $u\geq0$ and integrating on $B_{R}\setminus\bar{\Omega}$, we have that 
\begin{equation}\label{weilai'}
\int_{B_{R}\setminus\bar{\Omega}}u\Delta u\geq d\int_{B_{R}\setminus\bar{\Omega}}u.
 \end{equation} 
The left hand side can be computed by integration by part as 
\begin{equation*}
\int_{B_{R}\setminus\bar{\Omega}}u\Delta u=-\int_{B_{R}\setminus\bar{\Omega}}|Du|^{2}+\frac{d}{2}R^{2}|B_{R}|+cd|B_{R}|+O(R^{d-\alpha}).
\end{equation*}
Inserting the above equality into (\ref{weilai'}), we have that
\begin{equation}\label{chenke}
\int_{B_{R}\setminus\bar{\Omega}}|Du|^{2}\leq -d\int_{B_{R}\setminus\bar{\Omega}}u+\frac{d}{2}R^{2}|B_{R}|+cd|B_{R}|+O(R^{d-\alpha}).
 \end{equation} 
 
 Multiplying inequality (\ref{shenxia}) with $(x\cdot Du)\geq0$ and integrating on $B_{R}\setminus\bar{\Omega}$, we have that 
 \begin{equation}\label{zhuoyue'}
\int_{B_{R}\setminus\bar{\Omega}}(x\cdot Du)\Delta u\geq d\int_{B_{R}\setminus\bar{\Omega}}(x\cdot Du).
\end{equation}
The right hand side can be computed by integration by part as 
\begin{equation}
d\int_{B_{R}\setminus\bar{\Omega}}(x\cdot Du)=-d^{2}\int_{B_{R}\setminus\bar{\Omega}}u+\frac{1}{2}d^{2}R^{2}|B_{R}|+cd^{2}|B_{R}|+O(R^{d-\alpha}).\label{youxiu'--}
\end{equation}
Note that 
\begin{equation*}
(x\cdot Du)\Delta u=(x_{l}u_{l}u_{i})_{i}-\frac{1}{2}(x_{l}|Du|^{2})_{l}+\frac{d-2}{2}|Du|^{2}.
\end{equation*}
Then the left hand side of (\ref{zhuoyue'}) can be computed as 
\begin{equation*}
 \int_{B_{R}\setminus\bar{\Omega}}(x\cdot Du)\Delta u=\frac{d-2}{2}\int_{B_{R}\setminus\bar{\Omega}}|Du|^{2}-\frac{d}{2}|\Omega|+\frac{1}{2}dR^{2}|B_{R}|+O(R^{d-\alpha}).\nonumber
\end{equation*}
Inserting the above equality and (\ref{youxiu'--}) into (\ref{zhuoyue'}), we have that 
\begin{equation}\label{xuhoubao'}
\frac{d-2}{2}\int_{B_{R}\setminus\bar{\Omega}}|Du|^{2}\geq-d^{2}\int_{B_{R}\setminus\bar{\Omega}}u+\frac{1}{2}d(d-1)R^{2}|B_{R}|+cd^{2}|B_{R}|+\frac{d}{2}|\Omega|+O(R^{d-\alpha}).
\end{equation}
Combining (\ref{chenke}) and (\ref{xuhoubao'}) together, we have that 
\begin{equation*}
\int_{B_{R}\setminus\bar{\Omega}}u\geq \frac{d}{2(d+2)}R^{2}|B_{R}|+c|B_{R}|+\frac{1}{d+2}|\Omega|+O(R^{d-\alpha}).
\end{equation*}

We also know that 
\begin{align*}
&\quad\int_{B_{R}\setminus\bar{\Omega}}u\sigma_{k+1}(\lambda(D^{2}u))\\
&=-\frac{1}{k+1}\int_{B_{R}\setminus\bar{\Omega}}\sigma_{k+1}^{ij}u_{i}u_{j}+\frac{1}{(k+1)R}\int_{\partial B_{R}}u\sigma_{k+1}^{ij}u_{j}x_{i}\\
&=-\frac{1}{k+1}\int_{B_{R}\setminus\bar{\Omega}}H_{k}|Du|^{k+2}+\frac{1}{2}\binom{d}{k+1}R^{2}|B_{R}|+c\binom{d}{k+1}|B_{R}|+O(R^{d-\alpha})\\
&=-\frac{k+2}{2(k+1)}\binom{d}{k}\int_{B_{R}\setminus\bar{\Omega}}|Du|^{2}+\frac{d}{2(k+1)}\binom{d}{k}R^{2}|B_{R}|+c\binom{d}{k+1}|B_{R}|\\
&\ \ \ \   -\frac{1}{2(k+1)}\int_{\partial\Omega}H_{k-1}+O(R^{d-\alpha})\\
&\geq \frac{d(k+2)}{2(k+1)}\binom{d}{k}\int_{B_{R}\setminus\bar{\Omega}}u-\frac{kd}{4(k+1)}\binom{d}{k}R^{2}|B_{R}|-\frac{k(d+2)}{2(k+1)}\binom{d}{k}c|B_{R}|\\
&\ \ \ \  -\frac{1}{2(k+1)}\int_{\partial\Omega}H_{k-1}+O(R^{d-\alpha}).
\end{align*}
It follows that 
\begin{align*}
0&\geq \int_{B_{R}\setminus\bar{\Omega}}u\left(\sigma_{k+1}(\lambda(D^{2}u))-\binom{d}{k+1}\right)\\
&\geq \frac{k(d+2)}{2(k+1)}\binom{d}{k}\int_{B_{R}\setminus\bar{\Omega}}u-\frac{kd}{4(k+1)}\binom{d}{k}R^{2}|B_{R}|-\frac{k(d+2)}{2(k+1)}\binom{d}{k}c|B_{R}|\\
&\quad\quad-\frac{1}{2(k+1)}\int_{\partial\Omega}H_{k-1}+O(R^{d-\alpha})\\
&\geq\frac{1}{2(k+1)}\left[k\binom{d}{k}|\Omega|-\int_{\partial\Omega}H_{k-1}\right]+O(R^{d-\alpha})=O(R^{d-\alpha}).
\end{align*}
Letting $R\rightarrow\infty$, we have that $\sigma_{k+1}(\lambda(D^{2}u))=\binom{d}{k+1}$, which implies that $D^{2}u=I$. By (\ref{baozhang2}), we can conclude that $u(x)=\frac{1}{2}|x|^{2}-\frac{1}{2}$. The proof of Theorem \ref{ridig''} is finished.

\end{proof}

We secondly give the 
 \begin{proof}[proof of Theorem \ref{ridig}]
 
 We take $R>0$ sufficiently large such that $\bar{\Omega}\subset B_{R}$. The asymptotic behavior (\ref{fd4--}) implies that as $|x|\rightarrow\infty$, (\ref{chouba}) and (\ref{choudan}) hold.

Integrating the equation (\ref{la}) on $B_{R}\setminus\bar{\Omega}$, we have that 
\begin{equation}\label{yaduo---}
\int_{B_{R}\setminus\bar{\Omega}}\Delta u=d(|B_{R}|-|\Omega|).
\end{equation}
By the divergence theorem, (\ref{chouba}) and (\ref{baozhang2}), the left hand side can be computed as 
\begin{equation*}
\int_{B_{R}\setminus\bar{\Omega}}\Delta u=\frac{1}{R}\int_{\partial B_{R}}(x\cdot Du)-|\partial\Omega|=d|B_{R}|-|\partial\Omega|+O(R^{d-2-\alpha}).
\end{equation*}
Inserting the above equality into (\ref{yaduo---}) and letting $R\rightarrow\infty$, we have that 
\begin{equation*}
|\partial\Omega|=d|\Omega|.
\end{equation*}
 
 Multiplying the equation (\ref{la}) with $x\cdot Du$ and integrating on $B_{R}\setminus\bar{\Omega}$, we have that 
\begin{equation}\label{zhuoyue}
\int_{B_{R}\setminus\bar{\Omega}}(x\cdot Du)\Delta u=d\int_{B_{R}\setminus\bar{\Omega}}(x\cdot Du).
\end{equation}
By (\ref{baozhang2}), (\ref{fd4--}) and integration by part, the right hand side of (\ref{zhuoyue}) can be computed as 
\begin{equation}\label{youxiu}
d\int_{B_{R}\setminus\bar{\Omega}}(x\cdot Du)=-d^{2}\int_{B_{R}\setminus\bar{\Omega}}u+\frac{1}{2}d^{2}R^{2}|B_{R}|+cd^{2}|B_{R}|+O(R^{d-\alpha}).
\end{equation}
Note that 
\begin{equation*}
(x\cdot Du)\Delta u=(x_{l}u_{l}u_{i})_{i}-\frac{1}{2}(x_{l}|Du|^{2})_{l}+\frac{d-2}{2}|Du|^{2}.
\end{equation*}
Integrating the above equality on $B_{R}\setminus\bar{\Omega}$, by the divergence theorem, (\ref{baozhang2}) and (\ref{chouba}), we have that 
\begin{equation*}
\int_{B_{R}\setminus\bar{\Omega}}(x\cdot Du)\Delta u=\frac{d-2}{2}\int_{B_{R}\setminus\bar{\Omega}}|Du|^{2}-\frac{d}{2}|\Omega|+\frac{d}{2}R^{2}|B_{R}|+O(R^{d-\alpha}).
\end{equation*}
Inserting the above equality and (\ref{youxiu}) into (\ref{zhuoyue}), we have that 
\begin{equation}\label{xuhoubao}
\frac{d-2}{2}\int_{B_{R}\setminus\bar{\Omega}}|Du|^{2}+d^{2}\int_{B_{R}\setminus\bar{\Omega}}u=\frac{d(d-1)}{2}R^{2}|B_{R}|+cd^{2}|B_{R}|+O(R^{d-\alpha})+\frac{d}{2}|\Omega|.
\end{equation}

Multiplying the equation (\ref{la}) with $u$ and integrating on $B_{R}\setminus\bar{\Omega}$, we have that 
\begin{equation}\label{weilai}
\int_{B_{R}\setminus\bar{\Omega}}u\Delta u=d\int_{B_{R}\setminus\bar{\Omega}}u.
\end{equation}
By integration by part, (\ref{baozhang2}), (\ref{fd4--}) and (\ref{chouba}), we have that 
\begin{equation*}
\int_{B_{R}\setminus\bar{\Omega}}u\Delta u=-\int_{B_{R}\setminus\bar{\Omega}}|Du|^{2}+\frac{d}{2}R^{2}|B_{R}|+cd|B_{R}|+O(R^{d-\alpha}).
\end{equation*}
Inserting the above equality into (\ref{weilai}), we have that 
\begin{equation}\label{zhaohao}
\int_{B_{R}\setminus\bar{\Omega}}|Du|^{2}+d\int_{B_{R}\setminus\bar{\Omega}}u=\frac{d}{2}R^{2}|B_{R}|+cd|B_{R}|+O(R^{d-\alpha}).
\end{equation}

Solving from (\ref{xuhoubao}) and (\ref{zhaohao}), we have that 
\begin{equation*}
\int_{B_{R}\setminus\bar{\Omega}}|Du|^{2}=\frac{d}{d+2}R^{2}|B_{R}|-\frac{d}{d+2}|\Omega|+O(R^{d-\alpha}),
\end{equation*}
and 
\begin{equation}\label{luogu}
\int_{B_{R}\setminus\bar{\Omega}}u=\frac{d}{2(d+2)}R^{2}|B_{R}|+c|B_{R}|+\frac{1}{d+2}|\Omega|+O(R^{d-\alpha}).
\end{equation}

Since $2\sigma_{2}(\lambda(D^{2}u))=(\sigma_{2}^{ij}u_{i})_{j}$, by integration by part, (\ref{baozhang2}) and (\ref{choudan}), we have that 
\begin{align*}
\int_{B_{R}\setminus\bar{\Omega}}u\sigma_{2}(\lambda(D^{2}u))&=-\frac{1}{2}\int_{B_{R}\setminus\bar{\Omega}}\sigma_{2}^{ij}u_{i}u_{j}+\frac{1}{2R}\int_{\partial B_{R}}u\sigma_{2}^{ij}u_{j}x_{i}\\
&=-\frac{1}{2}\int_{B_{R}\setminus\bar{\Omega}}\sigma_{2}^{ij}u_{i}u_{j}+\frac{d(d-1)}{4}R^{2}|B_{R}|+\frac{d(d-1)}{2}c|B_{R}|+O(R^{d-\alpha}).
\end{align*}
Then we have that 
$$H_{1}=\frac{\sigma_{2}^{ij}u_{i}u_{j}}{|Du|^{3}},~~~~\Delta u=H_{1}|Du|+\frac{u_{ij}u_{i}u_{j}}{|Du|^{2}}.$$
Then we have that 
\begin{align*}
&\int_{B_{R}\setminus\bar{\Omega}}u\sigma_{2}(\lambda(D^{2}u))\\
&=-\frac{1}{2}\int_{B_{R}\setminus\bar{\Omega}}H_{1}|Du|^{3}+\frac{d(d-1)}{4}R^{2}|B_{R}|+\frac{d(d-1)}{2}c|B_{R}|+O(R^{d-\alpha})\\
&=-\frac{3}{4}d\int_{B_{R}\setminus\bar{\Omega}}|Du|^{2}+\frac{d^{2}}{4}R^{2}|B_{R}|+\frac{d(d-1)}{2}c|B_{R}|-\frac{1}{4}|\partial\Omega|+O(R^{d-\alpha})\\
&=\frac{3}{4}d^{2}\int_{B_{R}\setminus\bar{\Omega}}u-\frac{d^{2}}{8}R^{2}|B_{R}|-\frac{d(d+2)}{4}c|B_{R}|-\frac{1}{4}|\partial\Omega|+O(R^{d-\alpha}).
\end{align*}
It follows from the above inequality and (\ref{luogu}) that 
\begin{align*}
0&\geq \int_{B_{R}\setminus\bar{\Omega}}u\left(\sigma_{2}(\lambda(D^{2}u))-\frac{d(d-1)}{2}\right)\\
&= \frac{d(d+2)}{4}\int_{B_{R}\setminus\bar{\Omega}}u-\frac{d^{2}}{8}R^{2}|B_{R}|-\frac{d(d+2)}{4}c|B_{R}|-\frac{1}{4}|\partial\Omega|+O(R^{d-\alpha})\\
&=\frac{1}{4}(d|\Omega|-|\partial\Omega|)+O(R^{d-\alpha})=O(R^{d-\alpha}).
\end{align*}
Letting $R\rightarrow\infty$, we have that $\sigma_{2}(\lambda(D^{2}u))=\frac{d(d-1)}{2}$, which implies that $D^{2}u=I$. By (\ref{baozhang2}), we can conclude that $u(x)=\frac{1}{2}|x|^{2}-\frac{1}{2}$. The proof of Theorem \ref{ridig} is finished.

\end{proof}

\section{Proof of Lemma \ref{general}}

In this section, we give the 
\begin{proof}[proof of Lemma \ref{general}]

Without loss of generality, we may assume that $A=\mbox{diag}\{a_{1},\cdots,$ $ a_{d}\}$, $b=0$ and $c=0$, that is
\begin{equation*}
w(x):=u(x)-\frac{1}{2}x\cdot Ax=o(1), ~~~~\mbox{as }|x|\rightarrow\infty.
\end{equation*}
We take $R_{0}>0$ such that $\bar{\Omega}\subset B_{R_{0}}$ and 
\begin{equation}\label{ximenzi}
|w(x)|\leq \frac{1}{2},~~~~\forall~|x|>R_{0}.
\end{equation}
Since $\lambda(A)\in\Gamma_{k}$, we can take some $\varepsilon_{0}>0$ such that 
\begin{equation*}
(a_{1},\cdots,a_{d})-2\varepsilon_{0}(1,\cdots,1)\in\Gamma_{k}.
\end{equation*}
Let $\varphi(y):=\frac{1}{2}\sum\limits_{i=1}^{d}(a_{i}-2\varepsilon_{0})y_{i}^{2}$. For any $s>0$, define
\begin{equation*}
D_{s}:=\left\{y\in\RR^{d}:\frac{1}{2}\sum\limits_{i=1}^{d}a_{i}y_{i}^{2}<s+\varphi(y)\right\}=\left\{y\in\RR^{d}:\varepsilon_{0}|y|^{2}<s\right\}.
\end{equation*}

Fix $x\in\RR^{d}$ with $|x|>2R_{0}$. Let $R=\sqrt{2\varepsilon_{0}}|x|$. Let 
\begin{equation*}
u_{R}(y):=\left(\frac{4}{R}\right)^{2}u\left(x+\frac{R}{4}y\right),~~~~\forall~y\in D_{2},
\end{equation*}
and 
\begin{equation*}
w_{R}(y):=\left(\frac{4}{R}\right)^{2}w\left(x+\frac{R}{4}y\right),~~~~\forall~y\in D_{2}.
\end{equation*}
It follows that 
\begin{equation*}
w_{R}(y)=u_{R}(y)-\left(\frac{1}{2}y\cdot Ay+\frac{8}{R^{2}}x\cdot Ax+\frac{4}{R}x\cdot Ay\right),~~~~\forall~y\in D_{2}.
\end{equation*}
Let 
\begin{equation*}
\bar{u}_{R}(y):=u_{R}(y)-\left(\frac{8}{R^{2}}x\cdot Ax+\frac{4}{R}x\cdot Ay\right),~~~~\forall~y\in D_{2}.
\end{equation*}
By (\ref{baozhang3}) and (\ref{ximenzi}), we have that 
\begin{equation*}
\sigma_{k}(\lambda(D^{2}\bar{u}_{R}))=\binom{d}{k}~~~~\mbox{in }D_{2},
\end{equation*}
and 
\begin{equation}\label{liangdong}
\left|\bar{u}_{R}(y)-\frac{1}{2}y\cdot Ay\right|=|w_{R}(y)|\leq \frac{8}{R^{2}},~~~~\forall~y\in D_{2}.\end{equation}
In particular, 
\begin{equation*}
\|\bar{u}_{R}\|_{L^{\infty}(D_{2})}\leq C(A,R_{0},\varepsilon_{0}).
\end{equation*}

For $M\leq 2$, define 
\begin{equation*}
\Omega_{M,R}:=\{y\in D_{2}:\bar{u}_{R}(y)<M+\varphi(y)\}.
\end{equation*}
In view of (\ref{liangdong}), we have that 
\begin{equation*}
\bar{\Omega}_{1.5,R}\subset D_{1.6}
\end{equation*}
for any $R>R_{1}:=\max\{R_{0},10\}$. Applying the interior gradient estimate, Theorem 3.2 in \cite{cw2001}, to $\bar{u}_{R}$, we have that 
\begin{equation*}
\|D\bar{u}_{R}\|\leq C.
\end{equation*}
Here and in the following, $C\geq1$ denotes some constant depending on $k$, $d$, $A$, $R_{0}$. Applying the interior second derivatives estimate, Theorem 1.5 in \cite{cw2001}, to $\bar{u_{R}}$, we have that 
\begin{equation*}
(\bar{u}_{R}(y)-\varphi(y)-1.5)^{4}|D^{2}\bar{u}_{R}(y)|\leq C,~~~~\forall~y\in\Omega_{1.5,R}.
\end{equation*}
It follows that 
\begin{equation*}
\|D^{2}\bar{u}_{R}\|_{L^{\infty}(\Omega_{1.2,R})}\leq C,
\end{equation*}
and 
\begin{equation*}
D^{2}u_{R}=D^{2}\bar{u}_{R}\leq CI~~~~\mbox{in }D_{1.1}.
\end{equation*}
By the above inequality, the concavity of $\sigma_{k}^{1/k}$ and the interior estimate, we have that 
\begin{equation*}
\|u_{R}\|_{C^{4,\alpha}(\bar{D}_{1})}\leq C,~~~~\forall~\alpha\in(0,1).
\end{equation*}

It is easy to see that 
\begin{equation}\label{lomeng}
\|w_{R}\|_{C^{4,\alpha}(\bar{D}_{1})}\leq C~~\mbox{and}~~A+D^{2}w_{R}\leq CI~~~~\mbox{in }D_{1}.
\end{equation}
Clearly, $w_{R}$ satisfies that 
\begin{equation*}
a_{ij}^{R}D_{ij}w_{R}=0~~~~\mbox{in }D_{1},
\end{equation*}
where $$a^{ij}:=\int^1_0\sigma_k^{ij}(\lambda(A+sD^2w_R))ds.$$By (\ref{lomeng}), $a_{ij}^{R}$ is uniformly elliptic and 
\begin{equation*}
\|a_{ij}\|_{C^{2,\alpha}(\bar{D}_{1})}\leq C.
\end{equation*}
By the Schauder estimate, we have that 
\begin{equation*}
|D^{2}w_{R}(0)|\leq C\|w_{R}\|_{L^{\infty}(D_{1})}\leq CR^{-2}.
\end{equation*}
It follows that for $|x|>R_{1}$, 
\begin{equation*}
|D^{2}w(x)|=|D^{2}w_{R}(0)|\leq C|x|^{-2},
\end{equation*}
which implies that $D^{2}u(x)\rightarrow A\in\tilde{A}_{k}$ as $|x|\rightarrow\infty$.

\end{proof}
\bigskip

{\bf Acknowledgement.} The authors wish to thank  Professor Jiguang Bao for many helpful discussions. Part of  the work was done while the first author was visiting Fudan University. He would like to thank Fudan University for their hospitality.


\newcommand{\noopsort}[1]{}

\end{document}